\documentclass[a4paper,11pt,reqno]{amsart}

\usepackage{amssymb}
\usepackage[all]{xy}

\usepackage{amsmath}
\usepackage{amsthm}

\usepackage{hyperref}
\usepackage{xcolor}
\usepackage{enumitem}
\usepackage{subcaption}
\usepackage[final]{graphicx}
\usepackage{float}
\usepackage{epic}
\usepackage{setspace}

\usepackage[all]{xy}
\usepackage{color}
\usepackage{calligra}
\usepackage[T1]{fontenc}

\usepackage{quiver} 

\usepackage{verbatim}

\newcount\mycount

\numberwithin{equation}{section}

\numberwithin{equation}{section}
\newtheorem{introtheorem}{Theorem}

\newtheorem{introcorollary}[introtheorem]{Corollary}

\newtheorem{theorem}{Theorem}[section]

\newtheorem{lemma}[theorem]{Lemma}

\newtheorem{proposition}[theorem]{Proposition}

\newtheorem*{theorem*}{Theorem}

\theoremstyle{definition}
\newtheorem{definition}[theorem]{Definition}
\newtheorem*{definition*}{Definition}

\theoremstyle{remark}
\newtheorem{example}[theorem]{Example}

\theoremstyle{remark}

\newcommand{\Cs}{\ensuremath{\mathrm{C}^\ast}} 

\begin{document}

\title{Selfless reduced free products and graph products of \Cs-algebras}

\author{Felipe Flores}
\author{Mario Klisse}
\author{ M\'iche\'al \'O Cobhthaigh}
\author{Matteo Pagliero}

\address{Department of Mathematics, University of Virginia, Kerchof Hall. 141 Cabell Dr, Charlottesville, Virginia, United States}

\email{hmy3tf@virginia.edu}

\email{des5sf@virginia.edu}

\address{Department of Mathematics, Christian-Albrechts University Kiel,\\ Heinrich-Hecht-Platz 6,
Kiel, Germany}

\email{klisse@math.uni-kiel.de}

\address{Department of Mathematics and Computer Science, University of Southern Denmark, Campusvej 55,
Odense, Denmark}

\email{pagliero@imada.sdu.dk}

\date{\today. }

\begin{abstract}
Under mild assumptions, we show that reduced free products and reduced graph products of \Cs-algebras are completely selfless, without assuming the rapid decay property. In particular, our main theorems yield numerous new examples of simple, monotracial \Cs-algebras with strict comparison, stable rank one, and admitting a unique unital embedding of the Jiang--Su algebra $\mathcal{Z}$ up to approximate unitary equivalence, and of purely infinite \Cs-algebras in the traceless case. 
\end{abstract}

\maketitle

\section*{Introduction}

The Murray--von Neumann comparison of projections is fundamental to the theory of von Neumann algebras and the type classification of factors.
For general \Cs-algebras, the scarcity of projections (or lack thereof) requires an appropriate new version of the Murray--von Neumann comparison of projections. 
In \cite{Cuntz78}, Cuntz proposed to define a comparison relation for all positive elements with the intent of exploring dimension functions on simple \Cs-algebras.
For positive elements $a,b$ in a \Cs-algebra $A$, one says that $a$ is \emph{Cuntz subequivalent} to $b$ if there exists a sequence $(x_n)_{n\in \mathbb{N}} \subseteq  A$ such that $x_nbx_n^*\rightarrow a$.
This comparison is encoded by the partial order on what is nowadays known as the \emph{Cuntz semigroup} $\mathrm{Cu}(A)$ of $A$, a partially ordered Abelian monoid constructed out of the positive elements of the stabilization of $A$, i.e., $A\otimes\mathcal{K}$, and the Cuntz relation described above.
The object $\mathrm{Cu}(A)$ plays an important role in the classification of separable, simple, nuclear \Cs-algebras, known as the Elliott classification program (\cite{CarrionGabeSchafhauserTikuisisWhite23,GongLinNiu20,GongLinNiu20b,ElliottGongLinNiu15,ElliottGongLinNiu20,GongLin20,GongLin22,Schafhauser20,TikuisisWhiteWinter} to name a few of the latest contributions). Indeed, the Cuntz semigroup provides a finer invariant used by Toms to establish that $K$-theory and traces (and the pairing between them) alone are insufficient to classify all simple, nuclear, separable \Cs-algebras \cite{Toms}.
When the traces control this partial order in a suitable sense,  one says that the \Cs-algebra has \emph{strict comparison of positive elements} (or \emph{strict comparison} for short).
Consider for a moment a simple, unital, monotracial and exact \Cs-algebra $(A,\tau)$, and denote the canonical extension of $\tau$ to $A\otimes\mathcal{K}$ by $\tau$ as well.
If $d_{\tau}(b) := \lim_{n\to\infty}\tau(b^{1/n})$ for every $b\in(A\otimes\mathcal{K})_+$, then whenever $a$ is another positive element that is Cuntz subequivalent to $b$, one has that $d_{\tau}(a)\leq d_{\tau}(b)$.
Using the same notation as before, it is well-known that $(A,\tau)$ has strict comparison precisely when a partial converse holds: $d_{\tau}(a) < d_{\tau}(b)$ implies that $a$ is Cuntz subequivalent to $b$, for all $a,b\in(A\otimes\mathcal{K})_+$.
This property was first defined by Blackadar in \cite{Blackadar88}, and provides the best analogue of the Murray--von Neumann comparison of projections for factors.
Strict comparison is a key property for \Cs-algebras, and has deep connections with other regularity conditions;
most importantly with tensorial absorption of the Jiang--Su algebra $\mathcal Z$ \cite{JiangSu}.
A \Cs-algebra $A$ is said to be \emph{$\mathcal{Z}$-stable} if $A\cong A\otimes\mathcal{Z}$.
By the main result of \cite{Rordam04}, $\mathcal{Z}$-stable \Cs-algebras have strict comparison. 
It is worth emphasizing that $\mathcal{Z}$-stability is central to classify stably finite \Cs-algebras in the scope of the Elliott programme.

Consequently, it is desirable that this property is satisfied by many \Cs-algebras.
In particular, it is an open problem, and the missing piece to the famous Toms--Winter conjecture \cite{ElliottToms08,Winter10,Winter12,WinterZacharias10}, whether all separable, simple, nuclear, non-elementary \Cs-algebras  with strict comparison are $\mathcal{Z}$-stable.
However, the groundbreaking work of Matui and Sato \cite{MatuiSato12} established this under the additional assumptions of unitality and finitely many extremal traces, thereby granting access to classification via strict comparison.

Strict comparison is also deeply connected with stable rank one in the sense of Rieffel \cite{Rieffel} (see  \cite{Lin25}) and to several relevant problems, such as the classification of embeddings of the Jiang-Su algebra $\mathcal Z$ up to approximate unitary equivalence \cite{Robert12}, and the recent (negative) solution to the \Cs-algebraic Tarski problem \cite{ElayavalliSchafhauser25}. 

Nonetheless, finding criteria that ensure strict comparison for broader classes of (possibly non-nuclear) \Cs-algebras is not an easy task.
This difficulty is well illustrated by the fact that, until recently, it was an open problem whether the reduced group \Cs-algebra of the free group $\mathbb{F}_n$ has strict comparison when $2 \leq n < \infty$.
(The case of $\mathbb F_\infty$ was settled by Robert in \cite[Proposition 6.3.2]{Robert12}, where he attributes the result to R{\o}rdam.)
This problem has now been answered in the positive by Amrutam, Gao, Kunnawalkam Elayavalli, and Patchell in \cite{AmrutamGaoElayavalliPatchell24}, where they prove a much stronger property than strict comparison alone.
To be precise, they prove that finitely generated acylindrically hyperbolic groups with trivial finite radical and satisfying the rapid decay condition are \emph{\Cs-selfless}. In particular, their result implies that the corresponding reduced group \Cs-algebras are \emph{selfless} in the sense of Robert \cite{Robert25}, when considered as a \Cs-probability space equipped with the canonical trace.
Let us recall its definition and some consequences here.

\begin{definition*}[see {\cite[Definition 2.1]{Robert25}}]
Let $A$ be a unital \Cs-algebra equipped with a GNS-faithful state $\rho$. The pair $(A,\rho)$, referred to as a \emph{\Cs-probability space}, is said to be \emph{selfless} if $A\not\cong\mathbb{C}$ and the first factor embedding into the reduced free product \Cs-algebra\footnote{Here, and in the rest of the article, we denote the reduced free products of two unital \Cs-algebras $A$ and $B$ with respect to GNS-faithful states $\omega_A$ and $\omega_B$ by $(A,\omega_A)\star(B,\omega_B)$. The choice of notation is consistent with that of reduced graph products.} $\iota\colon (A,\rho) \hookrightarrow (A,\rho) \star (A,\rho)$ is \emph{existential}, i.e., there exists an ultrafilter $\mathcal{U}$ and an embedding $\theta\colon (A,\rho)\star (A,\rho) \hookrightarrow (A^{\mathcal{U}},\rho^{\mathcal{U}})$ such that the diagram
\[\begin{tikzcd}
	{(A,\rho)} && {(A^{\mathcal{U}},\rho^{\mathcal{U}})} \\
	& {(A,\rho)\star(A,\rho)}
	\arrow["\delta", hook, from=1-1, to=1-3]
	\arrow["{\iota}", hook, from=1-1, to=2-2]
	\arrow["\theta", hook, from=2-2, to=1-3]
\end{tikzcd}\]
commutes, where $\delta$ denotes the diagonal embedding.
\end{definition*}

When he introduced selfless \Cs-algebras, Robert was inspired by Dykema and R{\o}rdam's work on infinite reduced free products \cite{DykemaRordam} and their notion of eigenfree \Cs-probability spaces \cite{DykemaRordamB}.
As Robert observed, selfless \Cs-algebras (with a faithful trace) can be regarded as a \Cs-algebraic analogue of ${\rm II}_1$-factors.
Existential embeddings admit an equivalent model-theoretic formulation, although we take a more \Cs-algebraic viewpoint here.
Such embeddings are related to approximate splitting through the ultrapower (see \cite[Section 2]{GoldbringSinclair17}), and (for separable \Cs-algebras) are a special case of sequentially split $\ast$-homomorphisms in the sense of Barlak--Szabó (see \cite[Definition 2.1 and Theorem 4.19]{BarlakSzabo}).
In \cite[Theorem 2.9]{BarlakSzabo}, it is also established that the existence of a sequentially split $\ast$-homomorphisms between separable \Cs-algebras $A \to B$ forces that many structural properties of $B$, such as simplicity, stable rank one and strict comparison are passed to $A$, essentially because of their $\forall\exists$-axiomatizability.
For the same reason, if  a tracial \Cs-algebra $(A,\tau)$  is selfless, it must be simple, monotracial, and have strict comparison and stable rank one (see \cite[Theorem 3.1]{Robert25}).

Amrutam--Gao--Kunnawalkam Elayavalli--Patchell's work \cite{AmrutamGaoElayavalliPatchell24} has already inspired several generalisations and related results by other authors, which we now briefly recall. Selflessness has been established for reduced group \Cs-algebras associated with higher rank lattices \cite{Vigdorovich25}, free products of groups with subexponential decay \cite{ElayavalliPatchellTeryoshin25}, CAT(0) groups \cite{MaWangYang25}, and groups which satisfy a generalized Powers averaging property \cite{Ozawa25}. For twisted group \Cs-algebras of acylindrically hyperbolic groups with rapid decay, we refer the reader to \cite{RaumThielVilalta25} for results on strict comparison. When it comes to reduced free products of abstract \Cs-algebras, the state of the art result is the main theorem of \cite{HayesElayavalliRobert25} (and the follow-up work \cite{HayesElayavalliPatchellRobert25} about inclusions), which establishes that free products of \Cs-algebras with rapid decay and the \emph{Avitzour condition} are selfless. 
Originating in \cite{Avitzour82}, the Avitzour condition is the assumption that there exist unitaries $u,v_1,v_2$ as in Theorem \ref{MainTheorem2}.

In \cite{Ozawa25}, Ozawa introduced a strengthening of selflessness, termed \emph{complete selflessness}, which enjoys stability under spatial tensor products in the separable setting (see \cite[Theorem 11]{Ozawa25}).

\begin{definition*}[see {\cite[Section 6]{Ozawa25}}]
A \Cs-probability space\footnote{Recall that we assume the state $\rho$ is GNS-faithful.} $(A,\rho)$ is said to be \emph{completely selfless}
if $A\not\cong\mathbb{C}$ and the first factor embedding into the reduced free product \Cs-algebra $\iota\colon (A,\rho) \hookrightarrow (A,\rho) \star (A,\rho)$ is \emph{completely existential}, i.e., there exists an ultrafilter $\mathcal U$ and an embedding $\theta\colon(A,\rho)\star(A,\rho)\hookrightarrow(A^{\mathcal U},\rho^{\mathcal U})$ such that $\theta \circ \iota=\delta$ and, for every \Cs-algebra $C$, $\theta$ induces a $\ast$-homomorphism 
\[
((A,\rho)\star(A,\rho))\otimes C \longrightarrow ((A,\rho)\otimes C)^{\mathcal U} .
\]
\end{definition*}

Inspired by Ozawa's work on \Cs-probability spaces arising as reduced group \Cs-algebras \cite{Ozawa25}, we develop a new approach to obtain complete selflessness for certain natural classes of abstract \Cs-algebras.
The main result of the present work shows that reduced free products of \Cs-probability spaces are completely selfless if they satisfy the Avitzour condition.
We stress that we do not assume the rapid decay property, and thus obtain a full-fledged generalisation of Theorem B in \cite{HayesElayavalliRobert25}.
However, we remark that the rapid decay property remains relevant in the overarching picture.
In fact, in certain cases (see \cite[Theorem C]{HayesElayavalliRobert25}) it allows one to obtain selfless reduced free products without using the Avitzour condition.

In the following results, if  $\rho$ is a state on a \Cs-algebra $A$, we denote by $A^\rho$ its centralizer.

\begin{introtheorem}[{see Theorem \ref{Selflessness2}}]  \label{MainTheorem2}  Let $(A, \omega_{A})$ and $(B, \omega_{B})$ be unital \Cs-algebras equipped with GNS-faithful states $\omega_{A}$ and $\omega_{B}$, respectively. Assume that there exist unitaries 
\[
u \in \ker(\omega_{A}) \cap A^{\omega_{A}} \subseteq A
\quad \text{and} \quad
v_{1}, v_{2} \in \ker(\omega_{B}) \cap B^{\omega_{B}} \subseteq B
\]
such that $\omega_{B}(v_{1}^{\ast}v_{2}) = 0$. Then the free product \Cs-algebra $(A, \omega_{A}) \star (B, \omega_{B})$ is completely selfless. \end{introtheorem}

Assume now that $(A,\omega_A)$ and $(B,\omega_B)$ are as in Theorem \ref{MainTheorem2}, and that $\omega_A$ and $\omega_B$ are traces.
Then it follows from Avitzour's work \cite{Avitzour82} that $(A, \omega_{A}) \star (B, \omega_{B})$ is simple and satisfies the Dixmier property.
Moreover, in \cite{DykemaHaagerupRordam} Dykema, Haagerup and R{\o}rdam showed that $(A, \omega_{A}) \star (B, \omega_{B})$ has stable rank one, which in particular established that property for $\Cs_r(\mathbb{F}_2)$.

Our Theorem \ref{MainTheorem2} combined with \cite[Theorem 3.1]{Robert25}, \cite[Proposition 6.3.1]{Robert12}, and the very recent result of Gould \cite[Theorem 2]{Gould26}, implies all of the above, and the following: 

\begin{introcorollary}
    Let $(A, \omega_{A})$ and $(B, \omega_{B})$ be unital \Cs-algebras equipped with GNS-faithful states $\omega_{A}$ and $\omega_{B}$ satisfying the assumptions of Theorem \ref{MainTheorem2}.
    Then  $(A,\omega_A) \star (B,\omega_B)$ has the uniform Dixmier property (and thus simple). Moreover, one has the following dichotomy:
    \begin{enumerate}[leftmargin=*,label=\textup{(\roman*)}]
        \item  $\omega_A \star \omega_B$ is tracial, in which case it is the unique 2-quasitracial state, and $(A,\omega_A) \star (B,\omega_B)$ has strict comparison and admits a unique embedding of the Jiang--Su algebra $\mathcal Z$ up to approximate unitary equivalence; or
        \item $\omega_A \star \omega_B$ is not tracial, in which case  $(A,\omega_A) \star (B,\omega_B)$ is purely infinite.
    \end{enumerate}
\end{introcorollary}

We note that Hayes, Kunnawalkam Elayavalli, Patchell and Robert revisited our proof of Theorem \ref{MainTheorem2} in \cite[Theorem 3.4]{HayesElayavalliPatchellRobert25} to include selfless inclusions. 

Graph products of operator algebras have been introduced and studied by Mlotkowski \cite{Mlotkowski04}, Speicher--Wysocza\'{n}ski \cite{SpeicherWysoczanski16}, and Caspers--Fima \cite{CaspersFima17}. One may think of a reduced graph product of \Cs-algebras as a \Cs-algebra with an associated graph encoding the (non-)commutativity between its generating subalgebras. 
To be more precise, consider an undirected, simplicial  graph $\Gamma$ (i.e., edges have no direction, and there are no loops or multiple edges) with  set of vertices $V\Gamma$, and assume that each vertex $v\in V\Gamma$ is the label of a unital \Cs-algebra $A_v$, equipped with a GNS-faithful state $\omega_v$, and set $\mathbf{A}=(A_v)_{v\in\Gamma}$.
Then the \emph{reduced graph product \Cs-algebra} $(\mathbf{A}_{\Gamma},\omega_\Gamma)$ is a unital \Cs-algebra equipped with a GNS-faithful state that contains the vertex algebras $(A_v)_{v\in V\Gamma}$ in a canonical way, and such that two of these \Cs-subalgebras commute precisely when there is an edge between their labeling vertices (see Section \ref{GraphProductCAlgebras} for the precise definition).

Graph products unify both Voiculescu's free products (see \cite{Voiculescu85} and also \cite{Avitzour82}) and tensor products.
Indeed, if there are no edges in the graph $\Gamma$, the \Cs-algebras $(A_v)_{v\in V\Gamma}$ are in free position with respect to each other, and in fact $(\mathbf{A}_{\Gamma},\omega_\Gamma)=\star_{v\in V\Gamma} (A_v,\omega_v)$. On the other hand, if the graph is complete (i.e., every pair of vertices shares an edge), then $\mathbf{A}_{\Gamma}=\otimes_{v\in V\Gamma} A_v$. Moreover, $\mathbf{A}_\Gamma$ decomposes as a reduced free product over the connected component of $\Gamma$, and as a tensor product over the connected components of the complement graph $\Gamma^c$, obtained by taking as edge set the complement of the edge set of $\Gamma$ (see Section \ref{subsec:Graphs}).
In the second of our main results, we show that reduced graph products of \Cs-algebras are completely selfless provided that a suitable adaptation of the Avitzour condition is satisfied and that $\Gamma^c$  is connected.

\begin{introtheorem}[{see Theorem \ref{Selflessness1}}] \label{MainTheorem1} Let $\Gamma$ be an undirected, simplicial graph with at least three vertices,\footnote{Note that the graph with two vertices is covered by Theorem \ref{MainTheorem2}.} and let $\mathbf{A} = (A_{v})_{v \in V\Gamma}$ be a family of unital \Cs-algebras, each endowed with a GNS-faithful state $\omega_{v}$.  
Assume that the complement graph $\Gamma^{c}$ is connected and that every vertex $v \in V\Gamma$ admits a unitary 
$u_{v} \in A_{v}^{\omega_{v}} \cap \ker(\omega_{v})$.  
Then the reduced graph product \Cs-algebra $(\mathbf{A}_{\Gamma},\omega_\Gamma)$ is completely selfless. \end{introtheorem}

We remark that, assuming that the vertex \Cs-algebras are separable, with  Theorems \ref{MainTheorem2}, \ref{MainTheorem1}, and \cite[Theorem 11]{Ozawa25} we can handle general graphs $\Gamma$ by decomposing $\Gamma^c$ into its connected components.

As for reduced free products, by combining Theorem \ref{MainTheorem1} with \cite[Theorem 3.1]{Robert25}, \cite[Proposition 6.3.1]{Robert12}, and \cite[Theorem 2]{Gould26}, we obtain the following application. 
We remark that this is the first result establishing strict comparison and stable rank one for reduced graph product \Cs-algebras associated with graphs with more than two vertices.
We also extend previous results about simplicity, pure infiniteness, and the unique trace property of reduced graph products (see \cite[Theorem E]{Klisse25}).

\begin{introcorollary}
    Let $\Gamma$, $\mathbf{A}=(A_v)_{v\in V\Gamma}$, and $(\omega_{v})_{v\in V\Gamma}$ be as in Theorem \ref{MainTheorem1}.
    Then the reduced graph product $(\mathbf{A}_\Gamma,\omega_\Gamma)$ has the uniform Dixmier property. In particular, it is simple and either 
    \begin{enumerate}[leftmargin=*,label=\textup{(\roman*)}]
        \item  $\omega_\Gamma$ is tracial, in which case it is the unique 2-quasitracial state, and $(\mathbf{A}_\Gamma,\omega_\Gamma)$ has strict comparison, stable rank one, and admits a unique embedding of the Jiang--Su algebra $\mathcal Z$ up to approximate unitary equivalence; or
        \item $\omega_\Gamma$ is not tracial, in which case $(\mathbf{A}_\Gamma,\omega_\Gamma)$  is purely infinite.
    \end{enumerate}
\end{introcorollary}

We also remark that our results give access to new classes of examples.
In particular, pairs of \Cs-probability spaces that admit a Haar unitary clearly satisfy the Avitzour condition. Moreover, it follows from \cite[Corollary 5.5]{Thiel24} that a simple,  non-elementary \Cs-algebra equipped with a tracial state always has a Haar unitary. Together with Theorem \ref{MainTheorem2}, Theorem \ref{MainTheorem1} and \cite[Theorem 11]{Ozawa25}, this immediately yields the following consequence.

\begin{introcorollary}
Let $\Gamma$ be an undirected, simplicial graph with more than one vertex, and $\mathbf{A}=(A_v)_{v\in V\Gamma}$ a family of simple, separable, unital, non-elementary \Cs-algebras, each equipped with a GNS-faithful tracial state $\omega_v$. Assume that if $v$ is an isolated vertex in $\Gamma^c$, $(A_v,\omega_v)$ is completely selfless.
Then $(\mathbf{A}_\Gamma,\omega_\Gamma)$ is completely selfless.
\end{introcorollary}

In the non-tracial setting, \cite[Proposition 6.3]{Thiel24} provides conditions equivalent to admitting a Haar unitary. 

Note that, unlike previous works that only apply to reduced free products of reduced group \Cs-algebras \cite{AmrutamGaoElayavalliPatchell24,Ozawa25,HayesElayavalliPatchellRobert25}, our main result also applies to certain reduced free products of full group \Cs-algebras. For instance, for any $n \geq 2$, there exists a faithful tracial state $\tau_n$ on $\Cs(\mathbb F_n)$, obtained by adapting Choi's construction in \cite{Choi}. With respect to this trace, $(\Cs(\mathbb F_n),\tau_n)$ admits unitaries satisfying the Avitzour condition, so every reduced free product $(\Cs(\mathbb F_m),\tau_m)\star(\Cs(\mathbb F_n),\tau_n)$ with $m,n\geq2$ falls within the scope of Theorem~\ref{MainTheorem2}; see Example~\ref{Example} for details.  To the authors' knowledge, it remains unknown whether $(\Cs(\mathbb F_n),\tau_n)$ admits a filtration with rapid decay.

\subsection*{Outline of the main proofs} 
The proofs of Theorem~\ref{MainTheorem2} and Theorem~\ref{MainTheorem1}  share the same overarching methodology, which we explain in broad strokes below.

Our approach relies on a recent result of Ozawa, \cite[Theorem 13]{Ozawa25}, which establishes the following criterion ensuring that a \Cs-probability space $(A,\omega)$ is completely selfless. Assuming that $(A,\omega)$ is faithfully represented on $(B(\mathcal H),\rho)$, it is sufficient to find an ultrafilter $\mathcal U$ and an operator $T\in \mathcal B(\mathcal H)^{\mathcal U}$ such that 
\begin{enumerate}[leftmargin=*,label=\textup{(\roman*)}]
    \item $T^*T=1$,
    \item $T+T^* \in A^{\mathcal U}$,
    \item $TaT^*=\omega(a)$ for all $a\in A$, and
    \item $\rho^{\mathcal U}(aTT^*a^*)=0$ for all $a\in A$.
\end{enumerate}
Let us briefly comment on this criterion.
As shown by Ozawa, the conditions above ensure that the \Cs-algebra generated by $T$ in $\mathcal B(\mathcal H)^{\mathcal U}$ is a (state-preserving) copy of the Toeplitz \Cs-probability space, which is free from the diagonal copy of $A$ with respect to $\rho^{\mathcal U}$. Moreover, condition (ii) ensures that the subalgebra of the Toeplitz algebra generated by $T+T^*$, which is isomorphic to $C([-2,2])$, is contained in $A^{\mathcal U}$. Since one also has that the restriction of $\rho^{\mathcal U}$ to $ \Cs(T+T^*)$ is faithful, it follows that $(A,\omega)$ is selfless. To achieve complete selflessness, Ozawa observes that tensoring a \Cs-algebra to $A$ does not change the universal property (see \cite[Lemma 12]{Ozawa25}) used to provide a copy of the Toeplitz algebra.

As a consequence, the difficult part of our proofs is the construction of a suitable operator $T$. Our approach is combinatorial and proceeds as follows. First recall that the reduced graph product \Cs-algebra $\mathbf A_{\Gamma}$ is represented on the graph product Hilbert space $\mathcal H_{\Gamma}$, which is built from a vacuum vector $\Omega$, and tensor products of centered Hilbert spaces $\mathcal H^{\circ}_{\mathbf w}$ associated with reduced words ${\mathbf w}$ on the graph (see Section \ref{sec:Preliminaries-and-notation}). We exploit the connectedness and the size of the complement graph ($\Gamma^c$ for Theorem~\ref{MainTheorem2} or an auxiliary graph $\tilde{\Gamma}^c$ for Theorem~\ref{MainTheorem1}) to construct families of distinct closed walks that completely cover the complement graph.  Since every walk on the complement graph corresponds to a reduced word ${\mathbf w}$, each walk we found before dictates exactly how to multiply the Avitzour unitaries to obtain a given unitary element $u_{\mathbf w}$ in $\mathbf{A}_\Gamma$. We then use each unitary $u_{\mathbf w}$ to produce a projection $\mathbf{Q}(u_{\mathbf w}\Omega)$ onto the subspace of $\mathcal H_\Gamma$ generated by vectors whose tensor decomposition begins with $u_{\mathbf w}\Omega$.

The Avitzour condition forces a series of orthogonality relations for the projections $\mathbf{Q}(u_{\mathbf w}\Omega)$'s that are crucial in the next and final step of our proof. For each finite subset of minimal words on the graph, and for each $n\in\mathbb N$, we define an operator $T_{n,\mathcal{F}} \in \mathcal B(\mathcal H_\Gamma)$ as a linear combination of products of the abovementioned unitaries and projections. We then proceed to show that the operators $T_{n,\mathcal{F}}$'s asymptotically satisfy conditions (i)--(iv). Hence, for a suitable ultrafilter $\mathcal U$, one has that the limit along $\mathcal U$ of the $T_{n,\mathcal F}$'s  is the desired isometry $T\in\mathcal B(\mathcal H_\Gamma)$. By the discussion above, this is enough to conclude the proof.

\subsection*{Structure of the article} The article is organized as follows. Section \ref{sec:Preliminaries-and-notation} introduces preliminaries and fixes notation. In particular, it introduces reduced graph products of \Cs-algebras and Coxeter groups. Section \ref{sec:Main-construction} contains our main results. It is divided into three subsections: Subsection \ref{subsec:graph-product-construction}, which dives into the construction of certain (universal) \Cs-algebras associated with reduced graph products of \Cs-algebras, Subsection \ref{subsec:selflessness-graphs} which is dedicated to the proof of Theorem \ref{MainTheorem1} and Subsection \ref{subsec:selflessness-free}, which proves Theorem \ref{MainTheorem2}.


\section{Preliminaries and Notation\label{sec:Preliminaries-and-notation}}


\subsection{General notation}

We denote by $\mathbb{N} := \{0, 1, 2, \ldots\}$ the set of non-negative integers, and by $\mathbb{N}_{\geq 1} := \{1, 2, 3, \ldots\}$ the set of positive integers. The identity element of a group is denoted by $e$.

Throughout, inner products on Hilbert spaces are taken to be linear in the second variable. For a Hilbert space $\mathcal{H}$, we write $\mathcal{B}(\mathcal{H})$ for the C$^{*}$-algebra of all bounded linear operators on $\mathcal{H}$.

\subsection{Ultrapowers} 

Fix a free ultrafilter $\mathcal{U}$ on $\mathbb{N}$. For a C$^*$-algebra $A$, set 
\[ 
\ell^\infty(A) := \{(a_n)_{n\in\mathbb{N}} \subseteq A \mid \sup_{n\in\mathbb{N}} \|a_n\| < \infty\}, \qquad \mathcal{I}_{\mathcal{U}} := \{(a_n) \in \ell^\infty(A) \mid \lim_{n\to\mathcal{U}} \|a_n\| = 0\}, 
\] 
and define the \emph{ultrapower} of $A$ with respect to $\mathcal{U}$ by $A^{\mathcal{U}} := \ell^\infty(A) / \mathcal{I}_{\mathcal{U}}$. If $\rho$ is a state on $A$, its \emph{ultralimit} $\rho^{\mathcal{U}}$ on $A^{\mathcal{U}}$ is given by $\rho^{\mathcal{U}}([(a_n)_n]) = \lim_{n\to\mathcal{U}} \rho(a_n)$ for every sequence $(a_n)_n \in \ell^{\infty}(A)$, whose class in $A^{\mathcal U}$ is denoted by $[(a_n)_n]$.
The notation $(A,\rho)^{\mathcal U}$ refers to $(A^{\mathcal U},\rho^{\mathcal U})$.


\subsection{Graphs\label{subsec:Graphs}}

Given a graph $\Gamma$, its \emph{vertex set} is denoted by $V\Gamma$, and its \emph{edge set} by $E\Gamma$. Throughout this article, all graphs are assumed to be undirected, and \emph{simplicial}, that is, $E\Gamma \subseteq (V\Gamma \times V\Gamma) \setminus \{(v,v) \mid v \in V\Gamma\}$. A graph is said to be \emph{finite} if $V\Gamma$ contains finitely many vertices. 
We define the set of all (finite) \emph{words} in $V\Gamma$ by
\[
\mathcal{W}_{\Gamma} := \bigsqcup_{i \in \mathbb{N}} \underbrace{(V\Gamma \times \cdots \times V\Gamma)}_{i \text{ times}}.
\]
Such words are typically denoted by boldface letters. Note that, unlike in \cite{CaspersFima17}, we assume that $\mathcal{W}_{\Gamma}$ contains the empty word $\emptyset$.

Following \cite{Green90} and \cite{CaspersFima17}, we endow $\mathcal{W}_{\Gamma}$ with the \emph{shuffle equivalence} relation $\sim$, generated by the rule
\begin{equation}\label{eq:ShuffleEquivalence}
(v_{1},\ldots,v_{i-1},v_{i},v_{i+1},\ldots,v_{n})
\sim
(v_{1},\ldots,v_{i-1},v_{i+1},v_{i},\ldots,v_{n})
\quad \text{whenever } (v_{i},v_{i+1}) \in E\Gamma.
\end{equation}
Two words belonging to the same equivalence class under this relation are said to be \emph{shuffle equivalent}.

Similarly, for $\mathbf{v}, \mathbf{w} \in \mathcal{W}_{\Gamma}$, we write $\mathbf{v} \simeq \mathbf{w}$ and say that they are \emph{equivalent} if they belong to the same equivalence class generated by shuffle equivalence together with the additional relation that $(v_{1},\ldots,v_{i},v_{i+1},v_{i+2},\ldots,v_{n})$ is equivalent to $(v_{1},\ldots,v_{i},v_{i+2},\ldots,v_{n})$ whenever $v_{i}=v_{i+1}$.
\\

For a subset of vertices $V_0 \subseteq V\Gamma$, there is a unique graph $\Gamma_0$ whose vertex set is $V_0$ and whose edge set is specified as follows: for each pair $(v,w)\in V_0 \times V_0$, $(v,w)\in E\Gamma_0$ if and only if $(v,w)\in E\Gamma$. $\Gamma_0$ is said to be the \emph{full subgraph} of $\Gamma$ with vertex set $V_0$.
For a vertex $v \in V\Gamma$, the \emph{link} of $v$, denoted by $\mathrm{Link}(v)$, is the full subgraph of $\Gamma$ with vertex set $\{\, w \in V\Gamma \mid (v, w) \in E\Gamma \,\}$. The \emph{star} of $v$, denoted by $\mathrm{Star}(v)$, is the full subgraph of $\Gamma$ with vertex set $\{v\} \cup V(\mathrm{Link}(v))$.

A word $\mathbf{v} = (v_{1}, \ldots, v_{n}) \in \mathcal{W}_{\Gamma}$ is called \emph{reduced} if, for every pair of indices $1 \leq i < j \leq n$ with $v_{i} = v_{j}$, there exists an index $i < k < j$ such that $v_{k} \notin \mathrm{Star}(v_{i})$. We denote by $\mathcal{W}_{\mathrm{red}}$ the set of all reduced words.

The \emph{length} of a word $\mathbf{v} \in \mathcal{W}_{\Gamma}$, denoted by $|\mathbf{v}|$, is the length of the shortest representative in its equivalence class. Note that a word $\mathbf{v} = (v_{1}, \ldots, v_{n})$ is reduced if and only if $|\mathbf{v}| = n$.

By \cite[Lemma~1.3]{CaspersFima17}, any two equivalent reduced words $\mathbf{v} = (v_{1}, \ldots, v_{n})$ and $\mathbf{w} = (w_{1}, \ldots, w_{n})$ admit a unique permutation $\sigma$ of the set $\{1, \ldots, n\}$ such that
\[
\mathbf{w} = (v_{\sigma(1)}, \ldots, v_{\sigma(n)}) 
\quad \text{and} \quad 
\sigma(i) > \sigma(j) \quad \text{ if } i > j \text{ and } v_{i} = v_{j}.
\]

In view of the discussion above, we fix a subset $\mathcal{W}_{\mathrm{min}} \subseteq \mathcal{W}_{\mathrm{red}}$ consisting of representatives of the shuffle equivalence classes. The elements of $\mathcal{W}_{\mathrm{min}}$ are called \emph{minimal words}. Every word in $\mathcal{W}_{\Gamma}$ is equivalent to a unique minimal word.

For an undirected, simplicial graph $\Gamma$, we denote its \emph{complement} by $\Gamma^{c}$, which is the graph with vertex set $V\Gamma$ and edge set $\{\, (v, v') \in V\Gamma \times V\Gamma \mid v \neq v',\; (v, v') \notin E\Gamma \,\}$. The complement $\Gamma^{c}$ is again  undirected and simplicial.

A graph $\Gamma$ is said to be \emph{connected} if for every pair $(v,w)\in V\Gamma \times V\Gamma$ there exists a finite sequence of vertices $v_0,v_1,\dots,v_n \in V\Gamma$ with $v_0=v$ and $v_n=w$ such that $(v_{i-1},v_{i})\in E\Gamma$ for $i=1,\dots,n$.
A complete subgraph $\Gamma_0$ of $\Gamma$, which is a graph with $(v,w)\in E\Gamma_0$ for all $v,w\in V\Gamma_0$, is called a \emph{clique}. The \emph{degree} of a vertex $v\in V\Gamma$ is the number of elements $w\in V\Gamma$ with $(v,w)\in E\Gamma$.


\subsection{Coxeter groups\label{subsec:Right-angled-Coxeter-groups}}

A \emph{Coxeter group} is a group $W$ that admits a presentation of the form
\[
W = \left\langle S \,\middle|\, (st)^{m_{st}} = e \text{ for all } s, t \in S \right\rangle,
\]
where $S$ is an arbitrary generating set, and the exponents $m_{st} \in \{1, 2, \ldots, \infty\}$ satisfy $m_{ss} = 1$ and $m_{st} \ge 2$ for all distinct $s, t \in S$. A relation of the form $(st)^{m} = e$ is imposed when $m_{st} < \infty$; the case $m_{st} = \infty$ indicates that no such relation is imposed. The pair $(W, S)$ is referred to as a \emph{Coxeter system}.

A Coxeter system is said to be \emph{right-angled} if $m_{st} \in \{2, \infty\}$ for all $s \ne t$, that is, if any two distinct generators either commute or generate an infinite dihedral group. In a right-angled Coxeter group, if a cancellation of the form $s_{1} \cdots s_{n} = s_{1} \cdots \widehat{s_{i}} \cdots \widehat{s_{j}} \cdots s_{n}$ occurs for $s_{1}, \ldots, s_{n} \in S$, then $s_{i} = s_{j}$ and $s_{i}$ commutes with $s_{i+1}, \ldots, s_{j-1}$. For further background on Coxeter groups, we recommend consulting \cite{Davis08}.

Given a Coxeter system $(W, S)$, we denote by $|\cdot|$ the associated word length function. For elements $\mathbf{v}, \mathbf{w} \in W$, we say that $\mathbf{w}$ \emph{starts in} $\mathbf{v}$ if $|\mathbf{v}^{-1} \mathbf{w}| = |\mathbf{w}| - |\mathbf{v}|$, and write $\mathbf{v} \le_{R} \mathbf{w}$. Similarly, $\mathbf{w}$ is said to \emph{end in} $\mathbf{v}$ if
$|\mathbf{w} \mathbf{v}^{-1}| = |\mathbf{w}| - |\mathbf{v}|$, and we write $\mathbf{v} \le_{L} \mathbf{w}$. Both relations define partial orders on $W$, known respectively as the \emph{right weak Bruhat order} and the \emph{left weak Bruhat order}. For notational convenience, we shall usually write $\le$ instead of $\le_{R}$.

Given an undirected, simplicial graph  $\Gamma$,  we associate a right-angled Coxeter system $(W_{\Gamma}, S_{\Gamma})$ by setting
\[
S_{\Gamma} := V\Gamma
\quad \text{and} \quad
W_{\Gamma} := \left\langle S_{\Gamma} \,\middle|\, s^{2} = e \text{ for all } s \in S_{\Gamma},\; st = ts \text{ whenever } (s, t) \in E\Gamma \right\rangle.
\]
The group $W_{\Gamma}$ can be identified with the set $\mathcal{W}_{\mathrm{min}}$ of minimal words (equivalently, with $\mathcal{W}_{\mathrm{red}}$ modulo shuffle equivalence) via $(v_{1}, \ldots, v_{n}) \mapsto v_{1} \cdots v_{n}$. Under this identification, $\mathcal{W}_{\mathrm{min}}$ inherits a natural group structure. Moreover, the right and left weak Bruhat orders on $W_{\Gamma}$ induce corresponding partial orders on $\mathcal{W}_{\mathrm{min}}$. For notational convenience, we shall henceforth write $W_{\Gamma}$ in place of $\mathcal{W}_{\mathrm{min}}$ whenever no confusion arises. The length function $|\cdot|$ on $\mathcal{W}_{\mathrm{min}}$ from the previous subsection coincides with the word length on $W_{\Gamma}$ with respect to the generating set $S_{\Gamma}$.


\subsection{Reduced graph products of \texorpdfstring\Cs--algebras\label{GraphProductCAlgebras}}

Let $\Gamma$ be an undirected, simplicial graph, and consider a collection $\mathbf{A} := (A_{v})_{v\in V\Gamma}$ of unital \Cs-algebras, each carrying a GNS-faithful state $\omega_{v}$. We identify $A_{v}$ with its image inside $\mathcal{B}(\mathcal{H}_{v})$, where $\mathcal{H}_{v} := L^{2}(A_{v}, \omega_{v})$ is the associated GNS-Hilbert space and $\xi_{v}\in\mathcal{H}_{v}$ denotes the corresponding cyclic vector. For $x \in \mathcal{B}(\mathcal{H}_{v})$, define $x^{\circ} := x - \langle x\xi_{v}, \xi_{v}\rangle 1$, and let $A_{v}^{\circ} := \ker(\omega_{v})$. Denote by $\mathcal{H}_{v}^{\circ} := \mathcal{H}_{v} \ominus \mathbb{C}\xi_{v}$ the orthogonal complement of $\xi_{v}$. For a reduced word $\mathbf{v} = (v_{1}, \dots, v_{n}) \in \mathcal{W}_{\text{red}}$, we set $\mathcal{H}_{\mathbf{v}}^{\circ} := \mathcal{H}_{v_{1}}^{\circ} \otimes \cdots \otimes \mathcal{H}_{v_{n}}^{\circ}$.

As explained in Subsection~\ref{subsec:Right-angled-Coxeter-groups}, two equivalent reduced words $\mathbf{v} = (v_{1}, \dots, v_{n})$ and $\mathbf{w} = (w_{1}, \dots, w_{n})$ are related via a unique permutation $\sigma$ satisfying $\mathbf{w} = (v_{\sigma(1)}, \dots, v_{\sigma(n)})$ and $\sigma(i) > \sigma(j)$ whenever $i > j$ and $v_{i} = v_{j}$. This gives rise to a unitary operator $\mathcal{Q}_{\mathbf{v},\mathbf{w}} : \mathcal{H}_{\mathbf{v}}^{\circ} \to \mathcal{H}_{\mathbf{w}}^{\circ}$ via $\xi_{1}\otimes\cdots\otimes\xi_{n} \mapsto \xi_{\sigma(1)}\otimes\cdots\otimes\xi_{\sigma(n)}$.
For convenience, we suppress these unitaries from the notation and identify the spaces $\mathcal{H}_{\mathbf{v}}^\circ$ and $\mathcal{H}_{\mathbf{w}}^\circ $ through $\mathcal{Q}_{\mathbf{v},\mathbf{w}}$. Furthermore, we identify $\mathcal{W}_{\text{min}}$ with $W_{\Gamma}$, as in Subsection~\ref{subsec:Right-angled-Coxeter-groups}.

The \emph{graph product Hilbert space} (or \emph{Fock space}) associated with this data is
\[
\mathcal{H}_{\Gamma} := \mathbb{C}\Omega \oplus \bigoplus_{\mathbf{w}\in W_{\Gamma}\setminus\{e\}} \mathcal{H}_{\mathbf{w}}^{\circ},
\]
where $\Omega$ denotes the so-called \emph{vacuum vector}. Occasionally, we may write $\mathcal{H}_{e}^{\circ} := \mathbb{C}\Omega$.

For $v\in V\Gamma$, $x\in\mathcal{B}(\mathcal{H}_{v})$, $\mathbf{w}\in W_{\Gamma}$, and elementary tensors $\eta_{1}\otimes\cdots\otimes\eta_{n}\in\mathcal{H}_{\mathbf{w}}^{\circ}$ with $n=|\mathbf{w}|$, define
\begin{eqnarray*}
 &  & \lambda_{v}(x)(\eta_{1}\otimes\dots\otimes\eta_{n})\\
 & := & \begin{cases}
x^{\circ}\xi_{v}\otimes\eta_{1}\otimes\dots\otimes\eta_{n}+\langle x\xi_{v},\xi_{v}\rangle\eta_{1}\otimes\dots\otimes\eta_{n}, & \text{if }v\nleq\mathbf{w},\\
(x\eta_{1}-\langle x\eta_{1},\xi_{v}\rangle\xi_{v})\otimes\eta_{2}\otimes\dots\otimes\eta_{n}+\langle x\eta_{1},\xi_{v}\rangle\eta_{2}\otimes\dots\otimes\eta_{n}, & \text{if }v\leq\mathbf{w}\text{ and }\eta_{1}\in\mathcal{H}_{v}^{\circ}.
\end{cases}
\end{eqnarray*}
This gives a faithful, unital $*$-homomorphism $\lambda_{v}:\mathcal{B}(\mathcal{H}_{v})\to\mathcal{B}(\mathcal{H}_{\Gamma})$, and the images of $\lambda_{v}$ and $\lambda_{v'}$ commute whenever $(v,v')\in E\Gamma$ (see \cite[Subsection~2.1]{CaspersFima17}). The \emph{algebraic graph product} is defined as
\[
\star_{v,\Gamma}^{\mathrm{alg}}(A_{v},\omega_{v})
 := \ast\text{-alg}\bigl(\{\lambda_{v}(a)\mid v\in V\Gamma,\,a\in A_{v}\}\bigr)
 \subseteq \mathcal{B}(\mathcal{H}_{\Gamma}),
\]
and the corresponding \emph{reduced graph product \Cs-algebra} is its norm closure:
\[
\star_{v,\Gamma}(A_{v},\omega_{v})
 := \overline{\star_{v,\Gamma}^{\mathrm{alg}}(A_{v},\omega_{v})}^{\|\cdot\|}
 \subseteq \mathcal{B}(\mathcal{H}_{\Gamma}).
\]
We also set $\mathbf{A}_{\Gamma}^{\mathrm{alg}} := \star_{v,\Gamma}^{\mathrm{alg}}(A_{v},\omega_{v})$ and $\mathbf{A}_{\Gamma} := \star_{v,\Gamma}(A_{v},\omega_{v})$.

An element $a\in\mathbf{A}_{\Gamma}$ of the form $a = \lambda_{v_{1}}(a_{1})\cdots\lambda_{v_{n}}(a_{n})$ with $a_{i}\in A_{v_{i}}^{\circ}$ and $(v_{1},\dots,v_{n})\in\mathcal{W}_{\text{red}}$ is called \emph{reduced of type} $\mathbf{v}:=v_{1}\cdots v_{n}$, and $\mathbf{v}\in W_{\Gamma}$ is the \emph{associated word} of $a$. We will often write $a_{\mathbf{v}} := a_{1}\cdots a_{n} \in \mathbf{A}_{\Gamma}$, suppressing the maps $(\lambda_{v})_{v\in V\Gamma}$ for readability. For a reduced operator of type $\mathbf{v}\in W_{\Gamma}$, we refer to $|\mathbf{v}|$ as its \emph{length} and say that it \emph{starts} in $\mathbf{w}\in W_{\Gamma}$ if $\mathbf{w}\le\mathbf{v}$.

The \emph{graph product state} $\omega_{\Gamma}$ on $\mathcal B(\mathcal H_\Gamma)$ is defined as the vector state associated with $\Omega$. Slightly abusing notation, we denote the restriction of $\omega_{\Gamma}$ to $\mathbf{A}_{\Gamma}$ by $\omega_{\Gamma}$ as well, and we note that it is a GNS-faithful state. Moreover, $\omega_{\Gamma}(a)=0$ for every reduced element $a\in\mathbf{A}_{\Gamma}$ of type $\mathbf{v}\in W_{\Gamma}\setminus\{e\}$.

\begin{proposition}[see {\cite[Proposition 2.12]{CaspersFima17}}] Let $B$ be a unital \Cs-algebra with a GNS-faithful state $\omega$. Assume that for each $v\in V\Gamma$ there exists a unital faithful $*$-homomorphism $\pi_{v}:A_{v}\to B$ such that: 
\begin{itemize}[leftmargin=2em]
\item $B$ is generated by $\{\pi_{v}(a)\mid v\in V\Gamma,\,a\in A_{v}\}$. 
\item $\pi_{v}(A_{v})$ and $\pi_{v'}(A_{v'})$ commute whenever $(v,v')\in E\Gamma$. 
\item For each reduced operator $a=\lambda_{v_{1}}(a_{1})\dots\lambda_{v_{n}}(a_{n})$ with $a_{i}\in A_{v_{i}}^{\circ}$, one has $\omega(a)=0$. 
\end{itemize}
Then there exists a unique $*$-isomorphism $\pi:\mathbf{A}_{\Gamma}\to B$ with $\pi|_{A_{v}}=\pi_{v}$ for all $v\in V\Gamma$, and satisfying $\omega\circ\pi=\omega_{\Gamma}$. \end{proposition}


\section{Selflessness of graph product \texorpdfstring\Cs--algebras}\label{sec:Main-construction}

In this section, we establish Theorems~\ref{MainTheorem2} and~\ref{MainTheorem1}. Our argument partially relies on the construction of certain (universal) \Cs-algebras associated with reduced graph products of \Cs-algebras, a framework introduced and studied by the second author in \cite{Klisse25}. For completeness, we briefly recall this construction in Subsection~\ref{subsec:graph-product-construction}.


\subsection{\texorpdfstring\Cs--algebras from reduced graph products}\label{subsec:graph-product-construction}

Let $\Gamma$ be a finite, undirected, simplicial graph, and let $\mathbf{A}:=(A_{v})_{v\in V\Gamma}$ be a collection of unital \Cs-algebras, each equipped with a GNS-faithful state $\omega_{v}$. As in Subsection~\ref{GraphProductCAlgebras}, denote by $\mathcal{H}_{\Gamma}:=\mathbb{C}\Omega\oplus\bigoplus_{\mathbf{w}\in W_{\Gamma}\setminus\{e\}}\mathcal{H}_{\mathbf{w}}^{\circ}$ the associated graph product Hilbert space, and by $\star_{v,\Gamma}(A_{v},\omega_{v})$, or $\mathbf{A}_{\Gamma}$, the corresponding graph product \Cs-algebra.  
For each element $\mathbf{w}\in W_{\Gamma}$ (as defined in Subsection~\ref{subsec:Right-angled-Coxeter-groups}), let $Q_{\mathbf{w}}\in\mathcal{B}(\mathcal{H}_{\Gamma})$ denote the orthogonal projection onto the subspace $\bigoplus_{\mathbf{v}\in W_{\Gamma}\setminus\{e\}:\,\mathbf{w}\leq\mathbf{v}}\mathcal{H}_{\mathbf{v}}^{\circ}\subseteq\mathcal{H}_{\Gamma}$. We then define
\[
\mathfrak{A}(\mathbf{A},\Gamma)
:=C^{\ast}\!\left(\mathbf{A}_{\Gamma}\cup\{Q_{v}\mid v\in V\Gamma\}\right)
\subseteq\mathcal{B}(\mathcal{H}_{\Gamma}),
\]
as the \Cs-subalgebra of $\mathcal{B}(\mathcal{H}_{\Gamma})$ generated by $\mathbf{A}_{\Gamma}$, together with all projections $(Q_{v})_{v \in V\Gamma}$.  
Furthermore, let $\mathcal{D}(\mathbf{A},\Gamma)$ denote the \Cs-subalgebra of $\mathfrak{A}(\mathbf{A},\Gamma)$ consisting of all operators $x\in\mathfrak{A}(\mathbf{A},\Gamma)$ that are \emph{diagonal}, in the sense that
\[
x(\mathbb{C}\Omega)\subseteq\mathbb{C}\Omega
\quad\text{and}\quad
x(\mathcal{H}_{\mathbf{w}}^{\circ})\subseteq\mathcal{H}_{\mathbf{w}}^{\circ}
\quad\text{for every }\mathbf{w}\in W_{\Gamma}\setminus\{e\}.
\]
Slightly abusing notation, we denote by $\omega_{\Gamma}$ the restriction of the vacuum vector state to $\mathfrak{A}(\mathbf{A},\Gamma)$. As before, we suppress the $\ast$-embeddings $(\lambda_{v})_{v\in V\Gamma}$ and simply view each element $a\in A_{v}^{\circ}$, $v\in V\Gamma$, as an element of $\mathfrak{A}(\mathbf{A},\Gamma)$.

Following the terminology of \cite[Definition~2.6]{Klisse25}, for each $v\in V\Gamma$ and $a\in A_{v}$ we define:
\begin{itemize}[leftmargin=2em]
    \item the \emph{creation operator} associated with $a$ by    $a^{\dagger}:=Q_{v}aQ_{v}^{\perp}\in\mathfrak{A}(\mathbf{A},\Gamma)$,
    \item the \emph{diagonal operator} associated with $a$ by $\mathfrak{d}(a):=Q_{v}aQ_{v}\in\mathcal{D}(\mathbf{A},\Gamma)$,
    \item the \emph{annihilation operator} associated with $a$ by $((a^{*})^{\dagger})^{*}:=Q_{v}^{\perp}aQ_{v}\in\mathfrak{A}(\mathbf{A},\Gamma)$.
\end{itemize}

\begin{proposition}[see {\cite[Proposition~2.9]{Klisse25}}]
\label{DensityStatement}
Let $\Gamma$ be a finite, undirected, simplicial graph, and let $\mathbf{A}:=(A_{v})_{v\in V\Gamma}$ be a collection of unital \Cs-algebras with GNS-faithful states $(\omega_{v})_{v\in V\Gamma}$. Then the (dense) $\ast$-subalgebra of $\mathfrak{A}(\mathbf{A},\Gamma)$ generated by $\mathbf{A}_{\Gamma}$ and the projections $(Q_{v})_{v\in V\Gamma}$ coincides with
\begin{equation}
\mathrm{Span}\!\left\{
(a_{1}^{\dagger}\cdots a_{k}^{\dagger})\,d\,(b_{1}^{\dagger}\cdots b_{l}^{\dagger})^{*}
\ \middle|\
\begin{array}{l}
k,l\in\mathbb{N},\ (u_{1},\dots,u_{k}),(v_{1},\dots,v_{l})\in\mathcal{W}_{\mathrm{red}},\\
a_{i}\in A_{u_{i}}^{\circ},\ b_{j}\in A_{v_{j}}^{\circ},\ d\in\mathcal{D}_{0}(\mathbf{A},\Gamma)
\end{array}
\right\}.
\label{eq:DensitySet}
\end{equation}
Here $\mathcal{D}_{0}(\mathbf{A},\Gamma)\subseteq\mathcal{D}(\mathbf{A},\Gamma)$ denotes the set consisting of $1$ and all finite products $\mathfrak{d}(c_{1})\cdots\mathfrak{d}(c_{n})$ with $c_{i}\in A_{w_{i}}$, where $\{w_{1},\dots,w_{n}\}\subseteq V\Gamma$ forms a clique.
\end{proposition}

We call a non-zero operator of the form $x:=(a_{1}^{\dagger}\cdots a_{k}^{\dagger})\,d\,(b_{1}^{\dagger}\cdots b_{l}^{\dagger})^{\ast}$, where $k,l\in\mathbb{N}$ and $(u_{1},\ldots,u_{k}),(v_{1},\ldots,v_{l})\in\mathcal{W}_{\mathrm{red}}$, 
$a_{i}\in A_{u_{i}}^{\circ}$, $b_{j}\in A_{v_{j}}^{\circ}$ for $1\leq i\leq k$, $1\leq j\leq l$, 
and $d\in\mathcal{D}_{0}(\mathbf{A},\Gamma)$, \emph{elementary}; the collection of all such operators is denoted by $\mathcal{E}(\mathbf{A},\Gamma)$. By Proposition~\ref{DensityStatement}, the linear span of $\mathcal{E}(\mathbf{A},\Gamma)$ is dense in $\mathfrak{A}(\mathbf{A},\Gamma)$.  
Moreover, by \cite[Lemma~2.10]{Klisse25}, to each $x\in\mathcal{E}(\mathbf{A},\Gamma)$ we can associate a well-defined group element $\Sigma(x):=(u_{1}\cdots u_{k})(v_{1}\cdots v_{l})^{-1}\in W_{\Gamma}$,
which depends only on the operator $x$ itself and not on the particular choice of the elements $a_{i}$, $b_{j}$, or $d$.  
This element $\Sigma(x)$ is referred to as the \emph{signature} of $x$. Note that $x \mathcal{H}_{\mathbf{w}}^\circ \subseteq \mathcal{H}_{\Sigma(x)\mathbf{w}}^\circ $ for all $x\in \mathcal{E}(\mathbf{A},\Gamma)$, $\mathbf{w} \in W_\Gamma$.

An explicit decomposition of reduced operators in $\mathbf{A}_\Gamma$ into linear combinations of elementary
operators is provided in \cite[Proposition 2.6]{CaspersKlisseLarsen21}.


\subsection{The case \texorpdfstring{$\#V\Gamma>2$}{-}}

\label{subsec:selflessness-graphs}

In this subsection, we establish the following criterion for the complete selflessness of reduced graph product \Cs-algebras.

\begin{theorem}[{Theorem \ref{MainTheorem1}}] \label{Selflessness1} Let $\Gamma$ be an undirected, simplicial graph with $\#V\Gamma\geq3$ and let $\mathbf{A} = (A_{v})_{v \in V\Gamma}$ be a family of unital \Cs-algebras, each endowed with a GNS-faithful state $\omega_{v}$.  
Assume that the complement graph $\Gamma^{c}$ is connected and that every vertex $v \in V\Gamma$ admits a unitary 
$u_{v} \in A_{v}^{\omega_{v}} \cap \ker(\omega_{v})$.  
Then the reduced graph product \Cs-algebra $(\mathbf{A}_{\Gamma},\omega_\Gamma)$ is completely selfless. \end{theorem}

The proof of Theorem~\ref{Selflessness1} requires some preparation.

The following notion is adapted from \cite{Klisse23-2} and \cite{Klisse25}, where it was used in the study of the simplicity of reduced graph product \Cs-algebras. Implicitly, the concept was already employed in \cite{CaspersKlisseLarsen21} and \cite{Klisse23-1}.

\begin{definition} Let $\Gamma$ be an undirected, simplicial graph. A \emph{walk} in $\Gamma$ is a finite sequence of vertices $(v_{1},\dots,v_{n})\in V\Gamma\times\ldots\times V\Gamma$ with $(v_{i},v_{i+1})\in E\Gamma$ for all $i=1,\dots,n-1$. A walk is said to be \emph{closed} if $(v_{1},v_{n})\in E\Gamma$, and it is said to \emph{cover the whole graph} if $\{v_{1},\dots,v_{n}\}=V\Gamma$. \end{definition}

Note that a finite, undirected, simplicial graph $\Gamma$ admits a closed walk covering the whole graph if and only if it is connected.

\begin{lemma} \label{PathAbundance} Let $\Gamma$ be a finite, undirected, simplicial, connected graph with $\#V\Gamma\geq3$. Let $x,y\in V\Gamma$ with $(x,y)\in E\Gamma$, and let $n\in\mathbb{N}$. Then there exists an integer $L$ such that the number of distinct closed walks $(v_{1},\ldots,v_{L})\in V\Gamma\times\ldots\times V\Gamma$ satisfying $v_{1}=x$, $v_{L}=y$, and covering the whole graph is greater than $n$. \end{lemma}
\begin{proof}
Let $(u_{1},\ldots,u_{k})$ be any closed walk in $\Gamma$ with $u_{1}=x$, $u_{2}=y$, that covers the whole graph, and let $(v_{1},\ldots,v_{l})$ be a walk in $\Gamma$ starting at $v_{1}=x$ with $(v_{l},a),(v_{l},b)\in E\Gamma$ for distinct vertices $a,b\in V\Gamma$; such a walk exists because $\#\Gamma \geq 3$ implies that $\Gamma$ admits a vertex of degree at least 2. Choose an integer $K>n-1$ and set $L:=2(K+k+l-1)\in\mathbb{N}$. There are at least $n$ distinct walks of the form 
\[
(u_{1},\ldots,u_{k},v_{1},\ldots,v_{l},\underbrace{a,v_{l},\ldots,a,v_{l}}_{r_{1}\text{ times}},\underbrace{b,v_{l},\ldots,b,v_{l}}_{r_{2}\text{ times}},v_{l-1},\ldots,v_{1},u_{k},\ldots,u_{2}),
\]
with $r_{1}+r_{2}=K$. All these walks are closed, start in $x=u_{1}$, end in $y=u_{2}$, and cover the entire graph, which proves the claim. 
\end{proof}
Within the setting of Theorem~\ref{Selflessness1}, every reduced operator $a=a_{1}\cdots a_{n}\in\mathbf{A}_{\Gamma}$ of type $\mathbf{w}=s_{1}\cdots s_{n}\in W_{\Gamma}$ with $s_{i}\in S_{\Gamma}$, $a_{i}\in A_{s_{i}}^{\circ}$ naturally determines an orthogonal projection onto the direct sum of all subspaces of $\mathcal{H}_{\Gamma}$ of the form 
\[
\mathbb{C}(a\Omega)\otimes\mathcal{H}_{\mathbf{\mathbf{u}}}^{\circ}:=\left(\mathbb{C}(a_{1}\xi_{s_{1}})\otimes\cdots\otimes\mathbb{C}(a_{n}\xi_{s_{n}})\right)\otimes\mathcal{H}_{\mathbf{\mathbf{u}}}^{\circ},
\]
where $\mathbf{u}\in W_{\Gamma}$ is a group element with $|\mathbf{w}\mathbf{u}|=|\mathbf{w}|+|\mathbf{u}|$; we denote this projection by $\mathbf{Q}(a\Omega)$.

\begin{proposition} \label{DifferentPaths-1} Let $\Gamma$ be a finite, undirected, simplicial graph with $\#V\Gamma\geq3$, and let $\mathbf{A}:=(A_{v})_{v\in V\Gamma}$ be a family of unital \Cs-algebras, each endowed with a GNS-faithful state $\omega_{v}$. Assume that the complement $\Gamma^{c}$ is connected and that every vertex $v\in V\Gamma$ admits a unitary $u_{v}\in A_{v}^{\omega_{v}}\cap\ker(\omega_{v})$. Then, for every finite subset $\mathcal{F}\subseteq W_{\Gamma}\setminus\{e\}$, every vertex $v\in V\Gamma$, and every natural number $n\in\mathbb{N}$, there exist distinct closed walks 
\[
(v_{1}^{(1)},\ldots,v_{L}^{(1)}),\ldots,(v_{1}^{(n)},\ldots,v_{L}^{(n)})\in V\Gamma\times\ldots\times V\Gamma
\]
in $\Gamma^{c}$ with 
\[
v=v_{1}^{(1)}=\cdots=v_{1}^{(n)}\quad\text{and}\quad v_{L}^{(1)}=\cdots=v_{L}^{(n)},
\]
which cover the whole graph and satisfy $\mathbf{Q}(u_{\mathbf{g}_{i}}\Omega)\,a\,\mathbf{Q}(u_{\mathbf{g}_{j}}\Omega)=0$ for all $1\leq i,j\leq n$  and all reduced operators $a \in\mathbf{A}_{\Gamma}$ of type $\mathbf{w}\in\mathcal{F}$, where $\mathbf{g}_{i}:=v_{1}^{(i)}\cdots v_{L}^{(i)}$ and $\mathbf{g}_{j}:=v_{1}^{(j)}\cdots v_{L}^{(j)}$. \end{proposition}

\begin{proof} By Lemma~\ref{PathAbundance}, we can choose distinct closed walks $(\widetilde{v}_{1}^{(1)},\ldots,\widetilde{v}_{K}^{(1)}),\ldots,(\widetilde{v}_{1}^{(n)},\ldots,\widetilde{v}_{K}^{(n)})$ in $\Gamma^{c}$ covering the entire graph and satisfying $v=\widetilde{v}_{1}^{(1)}=\cdots=\widetilde{v}_{1}^{(n)}$ and $\widetilde{v}_{K}^{(1)}=\cdots=\widetilde{v}_{K}^{(n)}$. Let $R>\max_{\mathbf{w}\in\mathcal{F}}|\mathbf{w}|$ be an integer, and consider for each $1\leq i\leq n$ the walk $(v_{1}^{(i)},\ldots,v_{L}^{(i)})$ in $\Gamma^{c}$ with $L:=2(R+K-1)$ defined by 
\[
(\underbrace{\widetilde{v}_{1}^{(1)},\widetilde{v}_{2}^{(1)},\ldots,\widetilde{v}_{1}^{(1)},\widetilde{v}_{2}^{(1)}}_{R\text{ times}},\widetilde{v}_{3}^{(1)},\ldots,\widetilde{v}_{K}^{(1)},\widetilde{v}_{1}^{(i)},\ldots,\widetilde{v}_{K}^{(i)}).
\]
These walks are all distinct, have length $L$, are closed, and cover the whole graph $\Gamma^{c}$. Denote the element in $W_{\Gamma}$ associated with $(v_{1}^{(i)},\ldots,v_{L}^{(i)})$ by $\mathbf{g}_{i}$.

Now let $a=a_{1}\cdots a_{r}\in\mathbf{A}_{\Gamma}$ be a reduced operator of type $\mathbf{w}=s_{1}\cdots s_{r}\in\mathcal{F}\subseteq W_{\Gamma}\setminus\{e\}$ with $s_i \in S_\Gamma$ and $a_{i}\in A_{s_{i}}^{\circ}$. By \cite[Proposition 2.6]{CaspersKlisseLarsen21} (see also Proposition \ref{DensityStatement}) the operator $a$ can be expressed as a linear combination of elementary operators of the form $z:=(b_{1}^{\dagger}\cdots b_{k}^{\dagger})\,\left(\mathfrak{d}(c_{1})\cdots\mathfrak{d}(c_{l})\right)\,(d_{1}^{\dagger}\cdots d_{m}^{\dagger})^{\ast}\in\mathcal{E}(\mathbf{A},\Gamma)$, where $k,l,m\in\mathbb{N}$ with $k+l+m=r$, $(x_{1},\ldots,x_{k})$, $(y_{1},\ldots,y_{m})\in\mathcal{W}_{\mathrm{red}}$, and $b_{i}\in A_{x_{i}}^{\circ}$, $d_{i}\in A_{y_{i}}^{\circ}$, $c_{i}\in A_{w_{i}}^{\circ}$, where $\{w_{1},\ldots,w_{l}\}\subseteq V\Gamma$ forms a clique.
Moreover, by the same result  one can assume that 
\begin{equation} \label{eq:elementaryop}
    (x_1 \cdots x_k)(w_1 \cdots w_l)(y_1 \cdots y_m)^{-1} = \mathbf{w}.
\end{equation}

To prove the claim, it suffices to show that $\mathbf{Q}(u_{\mathbf{g}_{i}}\Omega)\,z\,\mathbf{Q}(u_{\mathbf{g}_{j}}\Omega)=0$ for all $1\leq i,j\leq n$ and $z$ as above. Consider vectors $(u_{\mathbf{g}_{i}}\Omega)\otimes\eta\in\mathbf{Q}(u_{\mathbf{g}_{i}}\Omega)\mathcal{H}_{\Gamma}$ and $(u_{\mathbf{g}_{j}}\Omega)\otimes\eta^{\prime}\in\mathbf{Q}(u_{\mathbf{g}_{j}}\Omega)\mathcal{H}_{\Gamma}$ with $\eta\in\mathcal{H}_{\mathbf{u}}^{\circ}$, $\eta^{\prime}\in\mathcal{H}_{\mathbf{u}^{\prime}}^{\circ}$, where $\mathbf{u},\mathbf{u}^{\prime}\in W_{\Gamma}$ are group elements with $|\mathbf{g}_{i}\mathbf{u}|=|\mathbf{g}_{i}|+|\mathbf{u}|$ and $|\mathbf{g}_{j}\mathbf{u}^{\prime}|=|\mathbf{g}_{j}|+|\mathbf{u}^{\prime}|$. Since $(b_{1}^{\dagger}\cdots b_{k}^{\dagger})^{\ast}$ and $(d_{1}^{\dagger}\cdots d_{m}^{\dagger})^{\ast}$ are products of annihilation operators, and $(v_{1}^{(i)},\ldots,v_{L}^{(i)})$ and $(v_{1}^{(j)},\ldots,v_{L}^{(j)})$ are closed walks that cover $\Gamma^{c}$, $\langle(u_{\mathbf{g}_{i}}\Omega)\otimes\eta,z((u_{\mathbf{g}_{j}}\Omega)\otimes\eta^{\prime})\rangle=0$ unless 
\begin{equation} \label{eq:wordscoincide}
x_{1}\cdots x_{k}=v_{1}^{(i)}\cdots v_{k}^{(i)}, \qquad \text{and} \qquad y_{1}\cdots y_{m}=v_{1}^{(j)}\cdots v_{m}^{(j)}.
\end{equation}
In that case, by the multiplication rules of annihilation operators (see, e.g., \cite[Lemma 2.8]{Klisse25}), we get that
\begin{eqnarray}
\nonumber
& & \left\langle (u_{\mathbf{g}_{i}}\Omega)\otimes\eta,z((u_{\mathbf{g}_{j}}\Omega)\otimes\eta^{\prime})\right\rangle \\
\nonumber
& & \qquad =  \overline{\left\langle (b_{1}^{\dagger}\cdots b_{k}^{\dagger})\Omega,(u_{v_{1}^{(i)}}\cdots u_{v_{k}^{(i)}})\Omega\right\rangle }\left\langle (d_{1}^{\dagger}\cdots d_{m}^{\dagger})\Omega,(u_{v_{1}^{(j)}}\cdots u_{v_{m}^{(j)}})\Omega\right\rangle \\
& & \qquad \qquad \times \left\langle (u_{v_{k+1}^{(i)}}\cdots u_{v_{L}^{(i)}}\Omega)\otimes\eta,\,\left(\mathfrak{d}(c_{1})\cdots\mathfrak{d}(c_{l})\right)\,((u_{v_{m+1}^{(j)}}\cdots u_{v_{L}^{(j)}}\Omega)\otimes\eta^{\prime})\right\rangle .\label{eq:ExpansionIdentity}
\end{eqnarray}

\vspace{1mm}
\noindent We proceed by distinguishing three cases.\\
\begin{itemize}[leftmargin=2em]
\item \emph{Case 1:} If $l=0$, then by convention $\mathfrak{d}(c_{1})\cdots\mathfrak{d}(c_{l})\in\mathbb{C}1$.  
Since $L-k > K$ and $L-m > K$, the last term in equation \eqref{eq:ExpansionIdentity} reduces to a scalar multiple of
\begin{eqnarray*}
&& \left\langle (u_{v_{k+1}^{(i)}}\cdots u_{v_{L}^{(i)}}\Omega)\otimes\eta,\, (u_{v_{m+1}^{(j)}}\cdots u_{v_{L}^{(j)}}\Omega)\otimes\eta^{\prime} \right\rangle \\
&& \qquad = \left\langle u_{v_{k+1}^{(i)}}\cdots u_{v_{L-K+1}^{(i)}}\cdots u_{v_{L}^{(i)}}\Omega\otimes\eta,\, u_{v_{m+1}^{(j)}}\cdots u_{v_{L-K+1}^{(j)}}\cdots u_{v_{L}^{(j)}}\Omega\otimes\eta^{\prime} \right\rangle.
\end{eqnarray*}
If $i\neq j$ and $k=m$, then by
\[
(v_{L-K+1}^{(i)},\ldots,v_{L}^{(i)})=(\widetilde{v}_{1}^{(i)},\ldots,\widetilde{v}_{K}^{(i)}) \neq (\widetilde{v}_{1}^{(j)},\ldots,\widetilde{v}_{K}^{(j)})=(v_{L-K+1}^{(j)},\ldots,v_{L}^{(j)}),
\]
it follows that the inner product above, and hence the expression in \eqref{eq:ExpansionIdentity}, vanishes.

Similarly, for arbitrary $1\leq i,j\leq n$, our construction guarantees that
\[
(v_{k+1}^{(i)},\ldots,v_{(2R+K-2)+(k-m)}^{(i)})\neq(v_{m+1}^{(j)},\ldots,v_{2R+K-2}^{(j)}) \quad \text{for } k<m,
\]
and
\[
(v_{k+1}^{(i)},\ldots,v_{2R+K-2}^{(i)})\neq(v_{m+1}^{(j)},\ldots,v_{(2R+K-2)+(m-k)}^{(j)}) \quad \text{for } m<k,
\]
which again implies that the inner product, and thus \eqref{eq:ExpansionIdentity}, vanishes. Indeed, for $k\not\equiv m\text{ (mod }2)$ we have that $v_{k+1}^{(i)}\neq v_{m+1}^{(j)}$, while for $k\equiv m\text{ (mod }2)$ equality of the expressions above would contradict the assumption that $(\widetilde{v}_{1}^{(1)},\ldots,\widetilde{v}_{K}^{(1)})$ covers the entire graph.

Finally, when $i=j$, we claim that one cannot have $m=k$. Indeed, if $m=k$, then by \eqref{eq:elementaryop} and \eqref{eq:wordscoincide} it follows that $\mathbf{w}=e$, contradicting our standing assumption that $\mathbf{w}\neq e$.
\item \emph{Case 2}: If $l=1$, and either $w_{1}\neq v_{k+1}^{(i)}$ or $w_{1}\neq v_{m+1}^{(j)}$, the term in \eqref{eq:ExpansionIdentity} again vanishes. If $w_{1}=v_{k+1}^{(i)}=v_{m+1}^{(j)}$, then
\begin{eqnarray*}
& &\left\langle (u_{\mathbf{g}_{i}}\Omega)\otimes\eta,z(u_{\mathbf{g}_{j}}\Omega)\otimes\eta^{\prime}\right\rangle \\
& & \qquad = \omega_{\Gamma}(u_{w_{1}}^{\ast}c_{1}u_{w_{1}})\overline{\left\langle (b_{1}^{\dagger}\cdots b_{k}^{\dagger})\Omega,(u_{v_{1}^{(i)}}\cdots u_{v_{k}^{(i)}})\Omega\right\rangle }\left\langle (d_{1}^{\dagger}\cdots d_{m}^{\dagger})\Omega,(u_{v_{1}^{(j)}}\cdots u_{v_{m}^{(j)}})\Omega\right\rangle \\
& & \qquad \qquad \times \left\langle (u_{v_{k+2}^{(i)}}\cdots u_{v_{L}^{(i)}}\Omega)\otimes\eta,\,(u_{v_{m+2}^{(j)}}\cdots u_{v_{L}^{(j)}}\Omega)\otimes\eta^{\prime}\right\rangle,
\end{eqnarray*}
which equals zero, since $\omega_{\Gamma}(u_{w_{1}}^{\ast}c_{1}u_{w_{1}})=\omega_{w_1}(u_{w_{1}}^{\ast}c_{1}u_{w_{1}})=\omega_{\Gamma}(c_{1})=0$.
\item \emph{Case 3}: For $l\geq 2$, since $\{w_{1},\ldots,w_{l}\}\subseteq V\Gamma$ forms a clique and $(v_{m+1}^{(j)},\ldots,v_{L}^{(j)})$ is a walk covering $\Gamma^{c}$, the expression in \eqref{eq:ExpansionIdentity} vanishes. This follows from the observation that the element in $W_{\Gamma}$ associated with $(v_{m+1}^{(j)},\ldots,v_{L}^{(j)})$ can only start in the letter $v_{m+1}^{(j)}$.
\end{itemize}

This completes the proof. \end{proof}

\begin{lemma} \label{MultiplicationIdentities-1} Let $\Gamma$ be a finite, undirected, simplicial graph with $\#V\Gamma\geq3$, and let $\mathbf{A}:=(A_{v})_{v\in V\Gamma}$ be a collection of unital \Cs-algebras, each equipped with a GNS-faithful state $\omega_{v}$. Assume that the complement $\Gamma^{c}$ is connected and that every vertex $v\in V\Gamma$ admits a unitary $u_{v}\in A_{v}^{\omega_{v}}\cap\ker(\omega_{v})$. Let $(v_{1},\ldots,v_{L}),(w_{1},\ldots,w_{L})\in V\Gamma\times\ldots\times V\Gamma$ be distinct closed walks covering the entire graph $\Gamma^{c}$, with $x:=v_{1}=w_{1}$ and $y:=v_{L}=w_{L}$, and set $\mathbf{g}:=v_{1}\cdots v_{L}$, $\mathbf{h}:=w_{1}\cdots w_{L}$.

For any $z\in V\Gamma\setminus\{x,y\}$ satisfying $(x,z)\notin E\Gamma$, the following statements hold:
\begin{enumerate}[leftmargin=*,label=\textup{(\roman*)}]
\item The following relations are satisfied: 
\begin{align}
\mathbf{Q}(u_{\mathbf{g}xy}\Omega)\mathbf{Q}(u_{\mathbf{h}xy}\Omega) & =0,\label{MultiplicationIdentities(1.1)}\\
\mathbf{Q}(u_{\mathbf{g}xy}\Omega)(u_{\mathbf{h}xz}u_{x}^*u_{\mathbf{h}xy}^{\ast})\mathbf{Q}(u_{\mathbf{h}xy}\Omega)^{\perp} & =0,\label{MultiplicationIdentities(1.2)}\\
\mathbf{Q}(u_{\mathbf{g}xy}\Omega)^{\perp}(u_{\mathbf{g}xy}u_{x}u_{\mathbf{g}xz}^{\ast})\mathbf{Q}(u_{\mathbf{h}xy}\Omega) & =0,\label{MultiplicationIdentities(1.3)}\\
\mathbf{Q}(u_{\mathbf{g}xy}\Omega)^{\perp}(u_{\mathbf{g}xy}u_{x}u_{\mathbf{g}xz}^{\ast})\mathbf{Q}(u_{\mathbf{g}xy}\Omega) & =0,\label{MultiplicationIdentities(1.4)}\\
\mathbf{Q}(u_{\mathbf{g}xy}\Omega)^{\perp}(u_{\mathbf{g}xy}u_{x}u_{\mathbf{g}xz}^{\ast})(u_{\mathbf{h}xz}u_{x}^{\ast}u_{\mathbf{h}xy}^{\ast})\mathbf{Q}(u_{\mathbf{h}xy}\Omega)^{\perp} & =0.\label{MultiplicationIdentities(1.5)}
\end{align}
\item The following inequalities hold: 
\begin{align}
(u_{\mathbf{g}xz}u_{x}^{\ast}u_{\mathbf{g}xy}^{\ast})\mathbf{Q}(u_{\mathbf{g}xy}\Omega)^{\perp}(u_{\mathbf{g}xy}u_{x}u_{\mathbf{g}xz}^{\ast}) & \leq\mathbf{Q}(u_{\mathbf{g}xz}\Omega),\label{MultiplicationIdentities(2.1)}\\
(u_{\mathbf{h}xz}u_{x}^{\ast}u_{\mathbf{h}xy}^{\ast})\mathbf{Q}(u_{\mathbf{h}xy}\Omega)^{\perp}(u_{\mathbf{h}xy}u_{x}u_{\mathbf{h}xz}^{\ast}) & \leq\mathbf{Q}(u_{\mathbf{h}xz}\Omega).\label{MultiplicationIdentities(2.2)}
\end{align}
\item We have 
\begin{align}
\mathbf{Q}(u_{\mathbf{g}xy}\Omega)^{\perp}(u_{\mathbf{g}xy}u_{x}u_{\mathbf{g}xz}^{\ast})=\mathbf{Q}(u_{\mathbf{g}xy}\Omega)^{\perp}(u_{\mathbf{g}xy}u_{x}u_{\mathbf{g}xz}^{\ast})\mathbf{Q}(u_{\mathbf{g}}\Omega).\label{MultiplicationIdentities(3.1)}
\end{align}
\end{enumerate}
\end{lemma}
\begin{proof}
Repeated applications of \cite[Lemma 2.8 (3)]{Klisse25} show that the elementary operator 
\[
(u_{v_{1}}^{\dagger}\cdots u_{v_{L}}^{\dagger}u_{x}^{\dagger}u_{y}^{\dagger})(u_{v_{1}}^{\dagger}\cdots u_{v_{L}}^{\dagger}u_{x}^{\dagger}u_{y}^{\dagger})^{\ast}\in\mathcal{E}(\mathbf{A},\Gamma)
\]
is a projection onto $\mathbf{Q}(u_{\mathbf{g}xy}\Omega)\mathcal{H}_{\Gamma}$. Therefore, it coincides with $\mathbf{Q}(u_{\mathbf{g}xy}\Omega)$ and $\Sigma(\mathbf{Q}(u_{\mathbf{g}xy}\Omega))=e$.

We claim that the following identity holds:
\begin{equation}
(u_{\mathbf{g}xz}u_{x}^{\ast}u_{\mathbf{g}xy}^{\ast})\mathbf{Q}(u_{\mathbf{g}xy}\Omega)^{\perp}=\mathbf{Q}(u_{\mathbf{g}xz}\Omega)(u_{\mathbf{g}xz}u_{x}^{\ast}u_{\mathbf{g}xy}^{\ast})\mathbf{Q}(u_{\mathbf{g}xy}\Omega)^{\perp} .\label{eq:GraphMainIdentity}
\end{equation}
Since $\mathbf{Q}(u_{\mathbf{g}xy}\Omega)$ is diagonal by \cite[Proposition 2.12]{Klisse25}, to show \eqref{eq:GraphMainIdentity} it suffices to verify that for every $\mathbf{w}\in W_{\Gamma}$ and every $\eta\in\mathcal{H}_{\mathbf{w}}^{\circ}\cap\ker(\mathbf{Q}(u_{\mathbf{g}xy}\Omega))$, the vector $(u_{\mathbf{g}xz}u_{x}^{\ast}u_{\mathbf{g}xy}^{\ast})\eta$ lies in $\mathbf{Q}(u_{\mathbf{g}xz}\Omega)\mathcal{H}_{\Gamma}$. For this, let $\mathbf{w}=s_{1}\cdots s_{n}$ be a reduced expression with $s_{i}\in S_\Gamma$, $|\mathbf{w}|=n$, and let $\eta=\eta_{1}\otimes\cdots\otimes\eta_{n}$ where $\eta_{i}\in\mathcal{H}_{s_{i}}^{\circ}$.

If $v_{1}\nleq\mathbf{w}$, then $(u_{\mathbf{g}xz}u_{x}^{\ast}u_{\mathbf{g}xy}^{\ast})\eta$ is clearly contained in $\mathbf{Q}(u_{\mathbf{g}xz}\Omega)\mathcal{H}_{\Gamma}$. So assume without loss of generality that $v_{1}\leq\mathbf{w}$ with $s_{1}=v_{1}$, so that
\begin{eqnarray*}
(u_{\mathbf{g}xz}u_{x}^{\ast}u_{\mathbf{g}xy}^{\ast})\eta &=& u_{\mathbf{g}xz}(u_{x}^{\ast}u_{y}^{\ast}u_{x}^{\ast}u_{v_{L}}^{\ast}\cdots u_{v_{2}}^{\ast}u_{v_{1}}^{\ast})(\eta_{1}\otimes\cdots\otimes\eta_{n}) \\
&=& u_{\mathbf{g}xz}(u_{x}^{\ast}u_{y}^{\ast}u_{x}^{\ast}u_{v_{L}}^{\ast}\cdots u_{v_{2}}^{\ast})\left((u_{v_{1}}^{\ast}\eta_{1}-\langle u_{v_{1}}\xi_{v_{1}},\eta_{1}\rangle\xi_{v_{1}}) \otimes \eta _2 \otimes\cdots\otimes\eta_{n}\right) \\
& & \qquad \qquad \qquad + \langle u_{v_{1}}\xi_{v_{1}},\eta_{1}\rangle u_{\mathbf{g}xz}(u_{x}^{\ast}u_{y}^{\ast}u_{x}^{\ast}u_{v_{L}}^{\ast}\cdots u_{v_{2}}^{\ast})\left(\eta_{2}\otimes\cdots\otimes\eta_{n}\right).
\end{eqnarray*}
The first summand lies in $\mathbf{Q}(u_{\mathbf{g}xz}\Omega)\mathcal{H}_{\Gamma}$ so that it remains to prove that the second summand does as well. As before, if $v_{1}v_{2}\nleq \mathbf{w}$, this is clearly the case. So we may without loss of generality assume that $v_{1}v_{2}\leq \mathbf{w}$ with $s_{2}=v_{2}$. Proceeding like this, we obtain that $(u_{\mathbf{g}xz}u_{x}^{\ast}u_{\mathbf{g}xy}^{\ast})\eta\in\mathbf{Q}(u_{\mathbf{g}xz}\Omega)\mathcal{H}_{\Gamma}$ whenever $n\leq L+2$, while for $n>L+2$ it remains to prove 
\[
\langle u_{y}\xi_{u_{y}},\eta_{L+2}\rangle\langle u_{x}\xi_{u_{x}},\eta_{L+1}\rangle\left(\prod_{i=1}^{L}\langle u_{v_{i}}\xi_{v_{i}},\eta_{i}\rangle\right)u_{\mathbf{g}xz}u_{x}^{\ast}(\eta_{L+3}\otimes\cdots\otimes\eta_{n})\in\mathbf{Q}(u_{\mathbf{g}xz}\Omega)\mathcal{H}_{\Gamma}.
\]
But this vector equals zero since $\eta\in\ker(\mathbf{Q}(u_{\mathbf{g}xy}\Omega))$, which completes the argument. This finishes the proof of \eqref{eq:GraphMainIdentity}.\\

\emph{About }(i): The identity in \eqref{MultiplicationIdentities(1.1)} follows directly from the fact that $\mathbf{g}xy$ and $\mathbf{h}xy$ correspond to distinct closed walks in $\Gamma^{c}$ covering the entire graph. Furthermore,  \eqref{eq:GraphMainIdentity} and \eqref{MultiplicationIdentities(1.1)}  yield
\[
\mathbf{Q}(u_{\mathbf{g}xy}\Omega)(u_{\mathbf{h}xz}u_{x}^{\ast}u_{\mathbf{h}xy}^{\ast})\mathbf{Q}(u_{\mathbf{h}xy}\Omega)^{\perp}=\mathbf{Q}(u_{\mathbf{g}xy}\Omega)\mathbf{Q}(u_{\mathbf{h}xz}\Omega)(u_{\mathbf{h}xy}u_{x}u_{\mathbf{h}xz}^{\ast})\mathbf{Q}(u_{\mathbf{h}xy}\Omega)^{\perp}=0,
\]
and
\[
\mathbf{Q}(u_{\mathbf{g}xy}\Omega)^{\perp}(u_{\mathbf{g}xy}u_{x}u_{\mathbf{g}xz}^{\ast})\mathbf{Q}(u_{\mathbf{h}xy}\Omega)=\mathbf{Q}(u_{\mathbf{g}xy}\Omega)^{\perp}(u_{\mathbf{g}xy}u_{x}u_{\mathbf{g}xz}^{\ast})\mathbf{Q}(u_{\mathbf{g}xz}\Omega)\mathbf{Q}(u_{\mathbf{h}xy}\Omega)=0,
\]
which establishes \eqref{MultiplicationIdentities(1.2)} and \eqref{MultiplicationIdentities(1.3)}. The identity in \eqref{MultiplicationIdentities(1.5)} follows in the same manner. Finally, \eqref{MultiplicationIdentities(1.4)} can be obtained by the same argument as for \eqref{eq:GraphMainIdentity}.\\

\emph{About }(ii): From \eqref{eq:GraphMainIdentity} we have 
\begin{eqnarray*}
& & (u_{\mathbf{g}xz}u_{x}^{\ast}u_{\mathbf{g}xy}^{\ast})\mathbf{Q}(u_{\mathbf{g}xy}\Omega)^{\perp}(u_{\mathbf{g}xy}u_{x}u_{\mathbf{g}xz}^{\ast}) \\
&=& \mathbf{Q}(u_{\mathbf{g}xz}\Omega)(u_{\mathbf{g}xz}u_{x}^{\ast}u_{\mathbf{g}xy}^{\ast})\mathbf{Q}(u_{\mathbf{g}xy}\Omega)^{\perp}(u_{\mathbf{g}xy}u_{x}u_{\mathbf{g}xz}^{\ast})\mathbf{Q}(u_{\mathbf{g}xz}\Omega) \\
&\leq& \mathbf{Q}(u_{\mathbf{g}xz}\Omega),
\end{eqnarray*}
and analogously,
\begin{eqnarray*}
& & (u_{\mathbf{h}xz}u_{x}^{\ast}u_{\mathbf{h}xy}^{\ast})\mathbf{Q}(u_{\mathbf{h}xy}\Omega)^{\perp}(u_{\mathbf{h}xy}u_{x}u_{\mathbf{h}xz}^{\ast}) \\
&=& \mathbf{Q}(u_{\mathbf{h}xz}\Omega)(u_{\mathbf{h}xz}u_{x}^{\ast}u_{\mathbf{h}xy}^{\ast})\mathbf{Q}(u_{\mathbf{h}xy}\Omega)^{\perp}(u_{\mathbf{h}xy}u_{x}u_{\mathbf{h}xz}^{\ast})\mathbf{Q}(u_{\mathbf{h}xz}\Omega) \\
&\leq& \mathbf{Q}(u_{\mathbf{h}xz}\Omega).
\end{eqnarray*}

\emph{About }(iii): The identity in \eqref{MultiplicationIdentities(3.1)} follows from \eqref{eq:GraphMainIdentity} in combination with $\mathbf{Q}(u_{\mathbf{g}xz}\Omega)\leq\mathbf{Q}(u_{\mathbf{g}}\Omega)$.\\

This completes the proof of the lemma. 
\end{proof}

We have now assembled all the ingredients necessary to establish Theorem~\ref{Selflessness1}.

\begin{proof}[{Proof of Theorem \ref{Selflessness1}}] We first consider the case where $\Gamma$ is a finite graph. For this, fix a vertex $x\in V\Gamma$ of degree at least $2$ in $\Gamma^{c}$. Let $\mathcal{F}\subseteq W_{\Gamma}\setminus\{e\}$ be a finite subset, and let $n\in\mathbb{N}_{\geq1}$. Choose distinct closed walks $(v_{1}^{(1)},\ldots,v_{L}^{(1)})$, $\ldots,(v_{1}^{(n)},\ldots,v_{L}^{(n)})$ in $\Gamma^{c}$ that cover the whole graph with $x=v_{1}^{(1)}=\cdots=v_{n}^{(1)}$ and $y:=v_{L}^{(1)}=\cdots=v_{L}^{(n)}$ as obtained in Proposition~\ref{DifferentPaths-1}. Set also $\mathbf{g}_{i}:=v_{1}^{(i)}\cdots v_{L}^{(i)}$, and choose a vertex $z\in V\Gamma\setminus\{x,y\}$ satisfying $(x,z)\notin E\Gamma$. We then introduce the operator 
\[
T_{n,\mathcal{F}}:=\frac{1}{\sqrt{2n}}\sum_{i=1}^{n}\left(\mathbf{Q}(u_{\mathbf{g}_ixy}\Omega)(u_{\mathbf{g}_ixy}u_{x}u_{\mathbf{g}_i xz}^{\ast})+(u_{\mathbf{g}_ixz}u_{x}^{\ast}u_{\mathbf{g}_ixy}^{\ast})\mathbf{Q}(u_{\mathbf{g}_ixy}\Omega)^{\perp}\right)\in\mathcal{B}(\mathcal{H}_{\Gamma}),
\]
so that $T_{n,\mathcal{F}}+T_{n,\mathcal{F}}^{\ast} =\frac{1}{\sqrt{2n}}\sum_{i=1}^{n}\left(u_{\mathbf{g}_{i}xy}u_{x}u_{\mathbf{g}_{i}xz}^{\ast}+u_{\mathbf{g}_{i}xz}u_{x}^{\ast}u_{\mathbf{g}_{i}xy}^{\ast}\right) \in\mathbf{A}_{\Gamma}$.

With Lemma~\ref{MultiplicationIdentities-1}~(i) we compute 
\begin{eqnarray*}
T_{n,\mathcal{F}}^{\ast}T_{n,\mathcal{F}} & = & \frac{1}{2n}\sum_{i,j=1}^{n}\left((u_{\mathbf{g}_{i}xz}u_{x}^{\ast}u_{\mathbf{g}_{i}xy}^{\ast})\mathbf{Q}(u_{\mathbf{g}_{i}xy}\Omega)+\mathbf{Q}(u_{\mathbf{g}_{i}xy}\Omega)^{\perp}(u_{\mathbf{g}_{i}xy} u_{x}u_{\mathbf{g}_{i}xz}^{\ast})\right)\\
 &  & \qquad\quad\times\left(\mathbf{Q}(u_{\mathbf{g}_{j}xy}\Omega)(u_{\mathbf{g}_{j}xy}u_{x}u_{\mathbf{g}_{j}xz}^{\ast})+(u_{\mathbf{g}_{j}xz}u_{x}^{\ast}u_{\mathbf{g}_{j}xy}^{\ast})\mathbf{Q}(u_{\mathbf{g}_{j}xy}\Omega)^{\perp}\right)\\
 & = & \frac{1}{2n}\sum_{i=1}^{n}\left((u_{\mathbf{g}_{i}xz}u_{x}^{\ast}u_{\mathbf{g}_{i}xy}^{\ast})\mathbf{Q}(u_{\mathbf{g}_{i}xy}\Omega)+\mathbf{Q}(u_{\mathbf{g}_{i}xy}\Omega)^{\perp}(u_{\mathbf{g}_{i}xy} u_{x}u_{\mathbf{g}_{i}xz}^{\ast})\right)\\
 &  & \qquad\quad\times\left(\mathbf{Q}(u_{\mathbf{g}_{i}xy}\Omega)(u_{\mathbf{g}_{i}xy}u_{x}u_{\mathbf{g}_{i}xz}^{\ast})+(u_{\mathbf{g}_{i}xz}u_{x}^{\ast}u_{\mathbf{g}_{i}xy}^{\ast})\mathbf{Q}(u_{\mathbf{g}_{i}xy}\Omega)^{\perp}\right)\\
 & = & \frac{1}{2n}\sum_{i=1}^{n}\left((u_{\mathbf{g}_{i}xz}u_{x}^{\ast}u_{\mathbf{g}_{i}xy}^{\ast})\mathbf{Q}(u_{\mathbf{g}_{i}xy}\Omega)(u_{\mathbf{g}_{i}xy}u_{x}u_{\mathbf{g}_{i}xz}^{\ast})+\mathbf{Q}(u_{\mathbf{g}_{i}xy}\Omega)^{\perp}\right)\\
 & = & 1-\frac{1}{2n}\sum_{i=1}^{n}\left((u_{\mathbf{g}_{i}xz}u_{x}^{\ast}u_{\mathbf{g}_{i}xy}^{\ast})\mathbf{Q}(u_{\mathbf{g}_{i}xy}\Omega)^{\perp}(u_{\mathbf{g}_{i}xy}u_{x}u_{\mathbf{g}_{i}xz}^{\ast})+\mathbf{Q}(u_{\mathbf{g}_{i}xy}\Omega)\right).
\end{eqnarray*}
By Lemma \ref{MultiplicationIdentities-1} (i), (ii), the projections $((u_{\mathbf{g}_{i}xz}u_{x}^{\ast}u_{\mathbf{g}_{i}xy}^{\ast})\mathbf{Q}(u_{\mathbf{g}_{i}xy}\Omega)^{\perp}(u_{\mathbf{g}_{i}xy}u_{x}u_{\mathbf{g}_{i}xz}^{\ast}))_{1\leq i\leq n}$ are pairwise orthogonal to each other, so that $P_{1}:=\sum_{i=1}^{n}(u_{\mathbf{g}_{i}xz}u_{x}^{\ast}u_{\mathbf{g}_{i}xy}^{\ast})\mathbf{Q}(u_{\mathbf{g}_{i}xy}\Omega)^{\perp}(u_{\mathbf{g}_{i}xy}u_{x}u_{\mathbf{g}_{i}xz}^{\ast})\in\mathcal{B}(\mathcal{H}_{\Gamma})$ is a projection as well. In the same way, $P_{2}:=\sum_{i=1}^{n}\mathbf{Q}(u_{\mathbf{g}_{i}xy}\Omega)\in\mathcal{B}(\mathcal{H}_{\Gamma})$ is a projection. It follows that 
\[
\left\Vert T_{n,\mathcal{F}}^{\ast}T_{n,\mathcal{F}}-1\right\Vert =\frac{\Vert P_{1}+P_{2}\Vert}{2n}\leq\frac{1}{n}.
\]

For every reduced operator $a\in\mathbf{A}_{\Gamma}$ of type $\mathbf{w}\in\mathcal{F}$, by Proposition~\ref{DifferentPaths-1} and Lemma~\ref{MultiplicationIdentities-1}~(iii) we have that 
\begin{eqnarray*}
 &  & T_{n,\mathcal{F}}^{\ast}aT_{n,\mathcal{F}}\\
 & = & \frac{1}{2n}\sum_{i,j=1}^{n}\left((u_{\mathbf{g}_{i}xz}u_{x}^{\ast}u_{\mathbf{g}_{i}xy}^{\ast})\mathbf{Q}(u_{\mathbf{g}_{i}xy}\Omega)+\mathbf{Q}(u_{\mathbf{g}_{i}xy}\Omega)^{\perp}(u_{\mathbf{g}_{i}xy} u_{x}u_{\mathbf{g}_{i}xz}^{\ast})\right) \\
 & & \quad \qquad \qquad \times a\left(\mathbf{Q}(u_{\mathbf{g}_{j}xy}\Omega)(u_{\mathbf{g}_{j}xy}u_{x}u_{\mathbf{g}_{j}xz}^{\ast})+(u_{\mathbf{g}_{j}xz}u_{x}^{\ast}u_{\mathbf{g}_{j}xy}^{\ast})\mathbf{Q}(u_{\mathbf{g}_{j}xy}\Omega)^{\perp}\right)\\
 & = & \frac{1}{2n}\sum_{i,j=1}^{n}\left((u_{\mathbf{g}_{i}xz}u_{x}^{\ast}u_{\mathbf{g}_{i}xy}^{\ast})\mathbf{Q}(u_{\mathbf{g}_{i}xy}\Omega)+\mathbf{Q}(u_{\mathbf{g}_{i}xy}\Omega)^{\perp}(u_{\mathbf{g}_{i}xy} u_{x}u_{\mathbf{g}_{i}xz}^{\ast})\right)\\
 &  & \quad \qquad\qquad\times(\mathbf{Q}(u_{\mathbf{g}_{i}}\Omega)a\mathbf{Q}(u_{\mathbf{g}_{j}}\Omega)) \\
 & & \quad  \qquad \qquad \qquad \times \left(\mathbf{Q}(u_{\mathbf{g}_{j}xy}\Omega)(u_{\mathbf{g}_{j}xy}u_{x}u_{\mathbf{g}_{j}xz}^{\ast})+(u_{\mathbf{g}_{j}xz}u_{x}^{\ast}u_{\mathbf{g}_{j}xy}^{\ast})\mathbf{Q}(u_{\mathbf{g}_{j}xy}\Omega)^{\perp}\right)\\
 & = & 0.
\end{eqnarray*}
Lastly, for every reduced operator $a\in\mathbf{A}_{\Gamma}$ of type $\mathbf{w}\in W_{\Gamma}$ with $|\mathbf{w}|<L$, we have $\mathbf{Q}(u_{\mathbf{g}_{i}}\Omega)(a^{\ast}\Omega)=0$ so that by Lemma~\ref{MultiplicationIdentities-1} (iii), $\omega_{\Gamma}(aT_{n,\mathcal{F}}T_{n,\mathcal{F}}^{\ast}a^*)=\Vert T_{n,\mathcal{F}}^{\ast}(a^*\Omega)\Vert_{2}^{2}=0$. Since the finite subset $\mathcal{F}\subseteq W_{\Gamma}\setminus\{e\}$ and the integer $n\in\mathbb{N}_{\geq1}$ were arbitrary, it follows that, for a suitable cofinal ultrafilter $\mathcal{U}$, the operators constructed above give rise to an isometry $T\in\mathcal{B}(\mathcal{H}_{\Gamma})^{\mathcal{U}}$ with $T+T^\ast \in \mathbf{A}_\Gamma^\mathcal{U}$, $T^{\ast}aT=\omega_{\Gamma}(a)$ and $\omega_\Gamma^{\mathcal U}(aTT^*a^*)=0$ for all $a\in\mathbf{A}_{\Gamma}$. By \cite[Theorem~13]{Ozawa25}, we conclude that $(\mathbf{A}_\Gamma,\omega_\Gamma)$ is completely selfless.\\

Now assume that $\Gamma$ is an infinite graph, and again fix a vertex $x\in V\Gamma$ of degree at least $2$ in $\Gamma^{c}$. For any finite subgraph $\Gamma_{0}\subseteq\Gamma$ obtained as the full subgraph of a finite subset of $V\Gamma$, we write $\mathbf{A}_{\Gamma_{0}}:=\star_{v,\Gamma_{0}}(A_{v},\omega_{v})$ and denote the corresponding graph product Hilbert space by $\mathcal{H}_{\Gamma_{0}}$. Since $\Gamma^c$ is connected, we can always replace $\Gamma_0$ with a larger finite graph whose complement graph is connected by adding a finite number of vertices to $V\Gamma_0$ and taking the corresponding full subgraph. For this reason, we may assume that $\Gamma_0$ has connected complement without loss of generality. By the first part of the proof, for every finite subgraph $\Gamma_{0}\subseteq\Gamma$ arising as the full subgraph of a finite subset of $V\Gamma$ with $x\in V\Gamma_0$, and with connected complement graph, every finite subset $\mathcal{F}\subseteq W_{\Gamma_{0}}\setminus\{e\}$, and every $n\in\mathbb{N}_{\geq1}$, there exist a natural number $L(\Gamma_{0},n,\mathcal{F})\in\mathbb{N}$ and an operator $T_{\Gamma_{0},n,\mathcal{F}}\in\mathcal{B}(\mathcal{H}_{\Gamma_{0}})\subseteq\mathcal{B}(\mathcal{H}_{\Gamma})$ such that $T_{\Gamma_{0},n,\mathcal{F}}^{\ast}aT_{\Gamma_{0},n,\mathcal{F}}=0$ for every reduced operator $a\in\mathbf{A}_{\Gamma_{0}}\subseteq\mathbf{A}_{\Gamma}$ of type $\mathbf{v}\in\mathcal{F}$, $\Vert T_{\Gamma_{0},n,\mathcal{F}}^{\ast}T_{\Gamma_{0},n,\mathcal{F}}-1\Vert\leq n^{-1}$, and $\omega_{\Gamma}(bT_{\Gamma_{0},n,\mathcal{F}}T_{\Gamma_{0},n,\mathcal{F}}^{\ast}b^{\ast})=0$ for every reduced operator $b\in\mathbf{A}_{\Gamma_{0}}\subseteq\mathbf{A}_{\Gamma}$ of type $\mathbf{w}\in W_{\Gamma_{0}}$ with $|\mathbf{w}|<L(\Gamma_{0},n,\mathcal{F})$. Without loss of generality, we can assume that $L$ is monotone in the sense that $L(\Gamma_{0}^{\prime},n^{\prime},\mathcal{F}^{\prime})\leq L(\Gamma_{0},n,\mathcal{F})$ for every full subgraph $\Gamma_{0}^{\prime}$ of $\Gamma$ with connected complement such that $V\Gamma_{0}' \subseteq V\Gamma_{0}$, $x\in V\Gamma_{0}^{\prime}$, every $n^{\prime}\leq n$, and every subset $\mathcal{F}^{\prime}\subseteq\mathcal{F}$. Since the union 
\[
\bigcup_{\substack{\Gamma_{0}\subseteq\Gamma\text{ finite with }x\in V\Gamma_{0}\\ \Gamma_0^c \text{ is connected}}}\text{Span}\left\{ a\in\mathbf{A}_{\Gamma_{0}}\mid a\text{ reduced of type }\mathbf{w}\in W_{\Gamma_{0}}\right\} 
\]
is dense in $\mathbf{A}_{\Gamma}$ (compare with the proof of \cite[Corollary 2.17]{CaspersFima17}).
It follows that, for a suitable cofinal ultrafilter $\mathcal{U}$, the operators constructed above give rise to an isometry $T\in\mathcal{B}(\mathcal{H}_{\Gamma})^{\mathcal{U}}$ with $T+T^{\ast}\in A_{\Gamma}^{\mathcal{U}}$, $T^{\ast}aT=\omega_{\Gamma}(a)$, and $\omega_{\Gamma}^{\mathcal{U}}(aTT^{*}a^{*})=0$ for all $a\in\mathbf{A}_{\Gamma}$. As before, we conclude from \cite[Theorem~13]{Ozawa25} that $(\mathbf{A}_{\Gamma},\omega_{\Gamma})$ is completely selfless.
 \end{proof}


\subsection{The case \texorpdfstring{$\#V\Gamma=2$}-}\label{subsec:selflessness-free}

For finite, undirected, simplicial graphs $\Gamma$ with $\#V\Gamma = 2$, the approach developed in the previous subsection requires adjustments, since an analogue of Lemma~\ref{PathAbundance} fails in this case. Indeed, Theorem~\ref{Selflessness1} cannot hold for graphs with only two vertices, as $\Cs_{r}(\mathbb{Z}_{2} \star \mathbb{Z}_{2})
\cong (\Cs_{r}(\mathbb{Z}_{2}), \tau) \star (\Cs_{r}(\mathbb{Z}_{2}), \tau)$ is not selfless, where $\tau$ denotes the canonical tracial state on $\Cs_{r}(\mathbb{Z}_{2})$.  
Nevertheless, in the present subsection we show that closely related ideas can be employed to handle reduced free products of \Cs-algebras satisfying \emph{Avitzour’s condition}. This substantially improves Theorem B of \cite{HayesElayavalliRobert25} by eliminating the rapid decay assumption.

\begin{theorem}[{Theorem \ref{MainTheorem2}}] \label{Selflessness2}
Let $(A, \omega_{A})$ and $(B, \omega_{B})$ be unital \Cs-algebras equipped with GNS-faithful states $\omega_{A}$ and $\omega_{B}$, respectively. Assume that there exist unitaries 
\[
u \in \ker(\omega_A) \cap A^{\omega_{A}} \subseteq A
\quad \text{and} \quad
v_{1}, v_{2} \in \ker(\omega_B)  \cap B^{\omega_{B}} \subseteq B
\]
such that $\omega_{B}(v_{1}^{\ast}v_{2}) = 0$. Then the free product \Cs-algebra $(A, \omega_{A}) \star (B, \omega_{B})$ is completely selfless.
\end{theorem}

As before, the proof of Theorem~\ref{Selflessness2} requires some preliminary results.

For notational convenience, for the rest of this section, we will denote by $\Gamma$ the finite, undirected, simplicial graph with vertex set $V\Gamma := \{1,2\}$ and edge set $E\Gamma := \emptyset$. We set $A_{1} := A$, $A_{2} := B$ and $\omega_{1} := \omega_{A}$, $\omega_{2} := \omega_{B}$. With this notation, $(A, \omega_{A}) \star (B, \omega_{B}) = \mathbf{A}_{\Gamma}$, and the associated right-angled Coxeter group $W_{\Gamma} \cong \mathbb{Z}_{2} \star \mathbb{Z}_{2}$ is the infinite dihedral group. As well, $\widetilde{\Gamma}$ will always refer to the finite, undirected, simplicial graph defined by $V\widetilde{\Gamma} := \{r, s, t\}$ and $E\widetilde{\Gamma} := \{(r, t), (t, r)\}$, so that $W_{\widetilde{\Gamma}} \cong \mathbb{Z}_{2} \star (\mathbb{Z}_{2} \times \mathbb{Z}_{2})$.
Finally, we also fix the notation for the following unitaries inside $\mathbf{A}_\Gamma$,
\[
u_r:=v_1^\ast, \qquad u_s:=u, \qquad u_t:=v_2^\ast,
\]
where $u,v_1$ and $v_2$ are as in the statement of Theorem~\ref{Selflessness2}.

We will often use without mention that, if one identifies $s\in V\widetilde{\Gamma}$ with $1\in V\Gamma$, and $r,t\in V\widetilde{\Gamma}$ with $2\in V\Gamma$, then any reduced word in $W_{\widetilde{\Gamma}}$ is also reduced when viewed in $ W_{\Gamma}$.

As in Subsection \ref{subsec:selflessness-graphs}, every reduced operator $a = a_{1} \cdots a_{n} \in \mathbf{A}_{\Gamma}$ of type $\mathbf{w} = s_{1} \cdots s_{n} \in W_{\Gamma}$ with $s_i \in S_\Gamma$, $a_{i} \in A_{w_{i}}^{\circ}$ determines an orthogonal projection $\mathbf{Q}(a\Omega)$ onto the direct sum of all subspaces of $\mathcal{H}_{\Gamma}$ of the form
\[
\mathbb{C}(a\Omega) \otimes \mathcal{H}_{\mathbf{j}}^{\circ}
:= 
\left(\mathbb{C}(a_{1}\xi_{s_{1}}) \otimes \cdots \otimes \mathbb{C}(a_{n}\xi_{s_{n}})\right) 
\otimes 
\mathcal{H}_{\mathbf{j}}^{\circ},
\]
where $\mathbf{j}$ is either empty or a finite alternating sequence in $\{1,2\}$, beginning in $\{1,2\} \setminus \{s_{n}\}$.

Analogously to the previous construction, consider a closed walk $(w_{1}, \ldots, w_{L}) \in V\widetilde{\Gamma} \times \cdots \times V\widetilde{\Gamma}$ in $\widetilde{\Gamma}^{c}$ with $w_{1} = s$ and $w_{L} = t$ that covers the entire graph.  
Define
\[
u_{\mathbf{g}} := u_{w_{1}} \cdots u_{w_{L}}, 
\qquad 
u_{\mathbf{g}st} := u_{w_{1}} \cdots u_{w_{L}} u_{s} u_{t}, 
\qquad 
u_{\mathbf{g}sr} := u_{w_{1}} \cdots u_{w_{L}} u_{s} u_{r},
\]
where $\mathbf{g} := w_{1} \cdots w_{L}$, and note that these are all reduced operators in $\mathbf{A}_\Gamma$.

\begin{proposition} \label{FreeDifferentPaths}
For every finite subset $\mathcal{F} \subseteq W_{\Gamma} \setminus \{e\}$ and every natural number $n \in \mathbb{N}$, there exist $n$ distinct closed walks
\[
(w_{1}^{(1)}, \ldots, w_{L}^{(1)}), \ldots, (w_{1}^{(n)}, \ldots, w_{L}^{(n)}) \in V\widetilde{\Gamma }\times \ldots \times V\widetilde{\Gamma}
\]
in $\widetilde{\Gamma}^{c}$ with
\[
s = w_{1}^{(1)} = \cdots = w_{1}^{(n)} \quad \text{and} \quad  t = w_{L}^{(1)} = \cdots = w_{L}^{(n)},
\]
which cover the whole graph, and satisfy $\mathbf{Q}(u_{\mathbf{g}_{i}}\Omega)\, a\, \mathbf{Q}(u_{\mathbf{g}_{j}}\Omega) = 0$ for all $1 \leq i,j \leq n$ and all reduced operators $a \in \mathbf{A}_{\Gamma}$ of type $\mathbf{w} \in \mathcal{F}$, where $\mathbf{g}_{i} := w_{1}^{(i)} \cdots w_{L}^{(i)}$ and $\mathbf{g}_{j} := w_{1}^{(j)} \cdots w_{L}^{(j)}$.
\end{proposition}

\begin{proof}
Choose $K \in \mathbb{N}$ such that $K > n$ and $K > \max_{\mathbf{w} \in \mathcal{F}} |\mathbf{w}|$. 
For each $1 \leq i \leq n$, consider the walk $(w_{1}^{(i)}, \ldots, w_{L}^{(i)})$ in $\widetilde{\Gamma}^{c}$ with $L := 2(2K + n)$ defined by
\begin{equation}
\bigl(
\underbrace{s, r, \ldots, s, r}_{K \text{ times}},
\underbrace{s, t, \ldots, s, t}_{K \text{ times}},
\underbrace{s, r, \ldots, s, r}_{i \text{ times}},
\underbrace{s, t, \ldots, s, t}_{n - i \text{ times}}
\bigr),
\label{eq:WalkDefinition}
\end{equation}
and denote the corresponding group element in $W_{\widetilde{\Gamma}}$ by $\mathbf{g}_{i} := w_{1}^{(i)} \cdots w_{L}^{(i)}$. Each of these walks has length $L$, is closed, and covers the whole graph.

Now let $a = a_{1} \cdots a_{m} \in \mathbf{A}_{\Gamma}$ be a reduced operator of type 
$\mathbf{w} = s_{1} \cdots s_{m} \in \mathcal{F} \subseteq W_{\Gamma} \setminus \{e\}$ with $s_i\in S_\Gamma$, $a_{i} \in A_{s_{i}}^{\circ}$. 
Then $a$ can be written as a linear combination of elementary operators (see, for instance, \cite[Proposition 2.6]{CaspersKlisseLarsen21}):
\begin{align}
\nonumber
a 
&= \sum_{k=0}^{m} 
  \bigl( (Q_{s_{1}} a_{1} Q_{s_{1}}^{\perp}) \cdots (Q_{s_{k}} a_{k} Q_{s_{k}}^{\perp}) \bigr)
  \bigl( (Q_{s_{k+1}}^{\perp} a_{k+1} Q_{s_{k+1}}) \cdots (Q_{s_{m}}^{\perp} a_{m} Q_{s_{m}}) \bigr) \\
\nonumber
&\quad
 + \sum_{k=0}^{m-1}
  \bigl( (Q_{s_{1}} a_{1} Q_{s_{1}}^{\perp}) \cdots (Q_{s_{k}} a_{k} Q_{s_{k}}^{\perp}) \bigr)
  (Q_{s_{k+1}} a_{k+1} Q_{s_{k+1}})
  \bigl( (Q_{s_{k+2}}^{\perp} a_{k+2} Q_{s_{k+2}}) \cdots (Q_{s_{m}}^{\perp} a_{m} Q_{s_{m}}) \bigr) \\
&= 
  \sum_{k=0}^{m} (a_{1}^{\dagger} \cdots a_{k}^{\dagger}) 
  \bigl( (a_{m}^{\ast})^{\dagger} \cdots (a_{k+1}^{\ast})^{\dagger} \bigr)^{\ast}
  + 
  \sum_{k=0}^{m-1} (a_{1}^{\dagger} \cdots a_{k}^{\dagger}) 
  \mathfrak{d}(a_{k+1}) 
  \bigl( (a_{m}^{\ast})^{\dagger} \cdots (a_{k+2}^{\ast})^{\dagger} \bigr)^{\ast},
\label{ElementaryDecomposition}
\end{align}
where every summand corresponding to an undefined index is assumed to be the identity.\\

We distinguish two cases.\\

\begin{itemize}[leftmargin=2em]
\item \emph{Case 1:} Suppose $m$ is even. 
Then either $s_{1}=1$ and $s_{m}=2$, or $s_{1}=2$ and $s_{m}=1$. 
In the first situation, the element $a u_{\mathbf{g}_{j}} \in \mathbf{A}_{\Gamma}$ is a reduced operator, and $a(\mathbf{Q}(u_{\mathbf{g}_{j}}\Omega)\mathcal{H}_{\Gamma})
  \subseteq \mathbf{Q}(a u_{\mathbf{g}_{j}}\Omega)\mathcal{H}_{\Gamma}$, so that $a\,\mathbf{Q}(u_{\mathbf{g}_{j}}\Omega)
  = \mathbf{Q}(a u_{\mathbf{g}_{j}}\Omega)\, a\, \mathbf{Q}(u_{\mathbf{g}_{j}}\Omega)$. Consequently,
\[
\mathbf{Q}(u_{\mathbf{g}_{i}}\Omega)\, a\, \mathbf{Q}(u_{\mathbf{g}_{j}}\Omega)
 = 
 \bigl(\mathbf{Q}(u_{\mathbf{g}_{i}}\Omega)\, \mathbf{Q}(a u_{\mathbf{g}_{j}}\Omega)\bigr)
 \, a \, \mathbf{Q}(u_{\mathbf{g}_{j}}\Omega)
 .
\]
To conclude the desired assertion, it is sufficient to show that $ \mathbf{Q}(u_{\mathbf{g}_{i}}\Omega)\, \mathbf{Q}(a u_{\mathbf{g}_{j}}\Omega)=0$. For this, fix vectors $(u_{\mathbf{g}_i}\Omega )\otimes\eta \in \mathbf{Q}(u_{\mathbf{g}_i}\Omega)\mathcal H_\Gamma$ and $(au_{\mathbf{g}_i}\Omega )\otimes\eta' \in \mathbf{Q}(au_{\mathbf{g}_i}\Omega)\mathcal H_\Gamma$ with $\eta \in \mathcal{H}_{\mathbf{f}}^{\circ}$, $\eta' \in \mathcal{H}_{\mathbf{f}'}^{\circ}$, where $\mathbf{f}$ and $\mathbf{f}'$ are either empty or finite alternating sequences in $\{1,2\}$ starting in $1$.
Observe that since $au_{\mathbf{g}_j}$ is reduced and $m<K$, we have that 
\begin{eqnarray*}
& & \left\langle (u_{w_1^{(i)}} \cdots u_{w_L^{(i)}}\Omega)\otimes\eta \, , \, (a_1 \dots a_m u_{w_1^{(j)}} \cdots u_{w_L^{(j)}}\Omega) \otimes\eta' \right\rangle \\
 & & \qquad = \left( \prod_{t=1}^{m}\langle u_{w_t^{(i)}}\Omega,a_t\Omega\rangle \right) \langle u_{w_{m+1}^{(i)}}\cdots u_{w_{L}^{(i)}}\Omega \otimes \eta,u_{w_1^{(j)}}\cdots u_{w_L^{(j)}}\Omega \otimes \eta^{\prime}\rangle \\
 &  & \qquad \qquad = \left( \prod_{t=1}^{m}\langle u_{w_t^{(i)}}\Omega,a_t\Omega\rangle \right) \left( \prod_{t=1}^{2K-m}\langle u_{w_{m+t}^{(i)}}\Omega,u_{w_t^{(j)}}\Omega\rangle \right) \\
 & & \qquad \qquad \qquad \qquad  \times \langle u_{w_{2K+1}^{(i)}}\cdots u_{w_{L}^{(i)}}\Omega\otimes\eta,u_{w_{2K-m+1}^{(j)}}\cdots u_{w_L^{(j)}}\Omega\otimes\eta^{\prime}\rangle.
\end{eqnarray*}
Furthermore, since $m$ is even, $w_{2K+2}=t$ and $w_{2K-m+2}=r$. Thus, using  the Avitzour condition and the assumption that $v_{1}$ and $v_{2}$ are contained in the centralizer of $\omega_{B}$,
\[
\langle u_{w_{2K+2}^{(i)}}\Omega,u_{w_{2K-m+2}^{(j)}}\Omega\rangle=\left\langle u_t\Omega,u_r\Omega\right\rangle=\omega_{\Gamma}(u_t^*u_r)=\omega_B(v_1^\ast v_2)=0.
\]

In the situation where $s_{1}=2$ and $s_{m}=1$, the equality $\mathbf{Q}(u_{\mathbf{g}_{i}}\Omega)\, a\, \mathbf{Q}(u_{\mathbf{g}_{j}}\Omega) = 0$ follows in the same way by taking adjoints and exchanging $i$ and $j$ in the argument above.

\item \emph{Case 2:} Suppose $m$ is odd. 
The same reasoning as above shows that 
$\mathbf{Q}(u_{\mathbf{g}_{i}}\Omega)\, a\, \mathbf{Q}(u_{\mathbf{g}_{j}}\Omega) = 0$
if $s_{1} = s_{m} = 2$. 
Hence, assume $s_{1} = s_{m} = 1$ and set $k_{0} := \frac{m - 1}{2}$. 
By the construction of the walks in~\eqref{eq:WalkDefinition} and the decomposition~\eqref{ElementaryDecomposition},
\begin{eqnarray}
 &  & \mathbf{Q}(u_{\mathbf{g}_{i}}\Omega)a\mathbf{Q}(u_{\mathbf{g}_{j}}\Omega)\nonumber \\
 & = & \mathbf{Q}(u_{\mathbf{g}_{i}}\Omega)\left(a_{1}^{\dagger}\cdots a_{k_{0}}^{\dagger}\right)\mathfrak{d}(a_{k_{0}+1})\left((a_{m}^{\ast})^{\dagger}\cdots(a_{k_{0}+2}^{\ast})^{\dagger}\right)^{\ast}\mathbf{Q}(u_{\mathbf{g}_{j}}\Omega)\label{eq:MultiplicationIdentity}
\end{eqnarray}
for all $1 \leq i,j \leq n$. Indeed, since $m$ is odd, every summand involving no non-trivial diagonal operator vanishes because it acts by shifting by an odd number, thus causing a mismatch in the alternating blocks of the walks. By a similar argument and by employing Avitzour's condition, every summand involving a diagonal operator with $k\neq k_{0}$ vanishes.

Now take vectors 
$(u_{\mathbf{g}_{i}}\Omega) \otimes \eta \in \mathbf{Q}(u_{\mathbf{g}_{i}}\Omega)\mathcal{H}_{\Gamma}$ and 
$(u_{\mathbf{g}_{j}}\Omega) \otimes \eta' \in \mathbf{Q}(u_{\mathbf{g}_{j}}\Omega)\mathcal{H}_{\Gamma}$ 
with $\eta \in \mathcal{H}_{\mathbf{f}}^{\circ}$, $\eta' \in \mathcal{H}_{\mathbf{f}'}^{\circ}$, where $\mathbf{f}$ and $\mathbf{f}'$ are either empty or finite alternating sequences in $\{1,2\}$ starting in $1$. 
A direct computation yields
\begin{eqnarray*}
& & \left\langle 
  (u_{\mathbf{g}_{i}}\Omega \otimes \eta),
  (a_{1}^{\dagger} \cdots a_{k_{0}}^{\dagger})
  \mathfrak{d}(a_{k_{0}+1})
  \bigl((a_{m}^{\ast})^{\dagger} \cdots (a_{k_{0}+2}^{\ast})^{\dagger}\bigr)^{\ast}
  (u_{\mathbf{g}_{j}}\Omega \otimes \eta')
\right\rangle \\
& & \qquad  = 
\omega_{\Gamma}(u_{w_{k_{0}+1}^{(i)}}^{\ast} a_{k_{0}+1} u_{w_{m - k_{0}}^{(j)}})\,
\overline{\omega_{\Gamma}(a_{k_{0}}^{\ast} \cdots a_{1}^{\ast} u_{w_{1}^{(i)}} \cdots u_{w_{k_{0}}^{(i)}})}\,
\omega_{\Gamma}(a_{k_{0}+2} \cdots a_{m} u_{w_{1}^{(j)}} \cdots u_{w_{m - k_{0} - 1}^{(j)}})\\
& & \quad \qquad \qquad  \times 
\left\langle 
 (u_{w_{k_{0}+2}^{(i)}} \cdots u_{w_{L}^{(i)}}\Omega) \otimes \eta,\,
 (u_{w_{m - k_{0} + 1}^{(j)}} \cdots u_{w_{L}^{(j)}}\Omega) \otimes \eta'
\right\rangle.
\end{eqnarray*}
By the construction of the walks, $w_{k_{0}+1}^{(i)} = w_{m - k_{0}}^{(j)}$, and hence
\[
\omega_{\Gamma}(u_{w_{k_{0}+1}^{(i)}}^{\ast} a_{k_{0}+1} u_{w_{m - k_{0}}^{(j)}})
 = \omega_{\Gamma}(a_{k_{0}+1}) = 0.
\]
Therefore, $\mathbf{Q}(u_{\mathbf{g}_{i}}\Omega)\, a\, \mathbf{Q}(u_{\mathbf{g}_{j}}\Omega) = 0$.
\end{itemize}
This completes the proof.
\end{proof}

The proof of the following lemma follows the same line of reasoning as that of Lemma~\ref{MultiplicationIdentities-1}.

\begin{lemma}\label{FreeMultiplicationIdentities}
Let $(w_{1}, \ldots, w_{L}), \ (w'_{1}, \ldots, w'_{L}) \in V\widetilde{\Gamma} \times \cdots \times V\widetilde{\Gamma}$ be distinct closed walks covering the whole graph $\widetilde{\Gamma}^{c}$, with $w_{1} = w'_{1} = s$, $w_{L} = w'_{L} = t$, and set $\mathbf{g} := w_{1} \cdots w_{L}$ and $\mathbf{h} := w'_{1} \cdots w'_{L}$. Then the following assertions hold:
\begin{enumerate}[leftmargin=*,label=\textup{(\roman*)}]

\item The following relations are satisfied:
\begin{align}
\mathbf{Q}(u_{\mathbf{g}st}\Omega)\,\mathbf{Q}(u_{\mathbf{h}st}\Omega) &= 0, \label{FreeMultiplicationIdentities(1.1)}\\
\mathbf{Q}(u_{\mathbf{g}st}\Omega)\,(u_{\mathbf{h}sr}u_{\mathbf{h}st}^{*})\,\mathbf{Q}(u_{\mathbf{h}st}\Omega)^{\perp} &= 0, \label{FreeMultiplicationIdentities(1.2)}\\
\mathbf{Q}(u_{\mathbf{g}st}\Omega)^{\perp}\,(u_{\mathbf{g}st}u_{\mathbf{g}sr}^{*})\,\mathbf{Q}(u_{\mathbf{h}st}\Omega) &= 0, \label{FreeMultiplicationIdentities(1.3)}\\
\mathbf{Q}(u_{\mathbf{g}st}\Omega)^{\perp}\,(u_{\mathbf{g}st}u_{\mathbf{g}sr}^{*})\,\mathbf{Q}(u_{\mathbf{g}st}\Omega) &= 0, \label{FreeMultiplicationIdentities(1.4)}\\
\mathbf{Q}(u_{\mathbf{g}st}\Omega)^{\perp}\,(u_{\mathbf{g}st}u_{\mathbf{g}sr}^{*})\,(u_{\mathbf{h}sr}u_{\mathbf{h}st}^{*})\,\mathbf{Q}(u_{\mathbf{h}st}\Omega)^{\perp} &= 0. \label{FreeMultiplicationIdentities(1.5)}
\end{align}

\item The following inequalities hold:
\begin{align}
(u_{\mathbf{g}sr}u_{\mathbf{g}st}^{*})\,\mathbf{Q}(u_{\mathbf{g}st}\Omega)^{\perp}\,(u_{\mathbf{g}st}u_{\mathbf{g}sr}^{*})
  &\le \mathbf{Q}(u_{\mathbf{g}s}\Omega), \label{FreeMultiplicationIdentities(2.1)}\\
(u_{\mathbf{h}sr}u_{\mathbf{h}st}^{*})\,\mathbf{Q}(u_{\mathbf{h}st}\Omega)^{\perp}\,(u_{\mathbf{h}st}u_{\mathbf{h}sr}^{*})
  &\le \mathbf{Q}(u_{\mathbf{h}s}\Omega). \label{FreeMultiplicationIdentities(2.2)}
\end{align}

\item We have
\begin{equation}
\mathbf{Q}(u_{\mathbf{g}st}\Omega)^{\perp}\,(u_{\mathbf{g}st}u_{\mathbf{g}sr}^{*})
  = \mathbf{Q}(u_{\mathbf{g}st}\Omega)^{\perp}\,(u_{\mathbf{g}st}u_{\mathbf{g}sr}^{*})\,\mathbf{Q}(u_{\mathbf{g}}\Omega).
\label{FreeMultiplicationIdentities(3.1)}
\end{equation}

\end{enumerate}
\end{lemma}

\begin{proof} We claim that the following identity holds:
\begin{equation}\label{eq:FreeMainIdentity}
 (u_{\mathbf{g}sr}u_{\mathbf{g}st}^{*})\,\mathbf{Q}(u_{\mathbf{g}st}\Omega)^{\perp}
 = \mathbf{Q}(u_{\mathbf{g}s}\Omega)\,(u_{\mathbf{g}sr}u_{\mathbf{g}st}^{*})\,\mathbf{Q}(u_{\mathbf{g}st}\Omega)^{\perp}.
\end{equation}
As in the proof of Lemma \ref{MultiplicationIdentities-1} one obtains that $\mathbf{Q}(u_{\mathbf{g}st}\Omega)$ is a diagonal operator. It therefore suffices to verify that for every $\mathbf{w}\in W_{\Gamma}$ and every $\eta \in \mathcal{H}_{\mathbf{w}}^{\circ} \cap \ker(\mathbf{Q}(u_{\mathbf{g}st}\Omega))$,
the vector $(u_{\mathbf{g}sr}u_{\mathbf{g}st}^{*})\eta$ lies in $\mathbf{Q}(u_{\mathbf{g}s}\Omega)\mathcal{H}_{\Gamma}$. For this, let $\mathbf{w} = s_{1}\cdots s_{n}$ be a reduced expression with $s_{i}\in V\Gamma$, $|\mathbf{w}|=n$, and let $\eta = \eta_{1}\otimes\cdots\otimes\eta_{n}$ where $\eta_{i}\in\mathcal{H}_{s_{i}}^{\circ}$.

We distinguish two cases.

\begin{itemize}[leftmargin=2em]

\item \emph{Case 1:} If $s_{1}=2$, we obtain by employing the Avitzour condition in the second step, that
\begin{eqnarray*}
 (u_{\mathbf{g}sr}u_{\mathbf{g}st}^{*})\,\eta
 &=& u_{\mathbf{g}sr}\left((u_{t}^{\ast}\xi_{2})\otimes(u_{s}^{\ast}\xi_{1})\otimes(u_{w_{L}}^{\ast}\xi_{2})\otimes\cdots\otimes(u_{w_{1}}^{\ast}\xi_{1})\otimes\eta\right) \\
 &=& u_{\mathbf{g}s}\left((u_{r}u_{t}^{\ast}\xi_{2})\otimes(u_{s}^{\ast}\xi_{1})\otimes(u_{w_{L}}^{\ast}\xi_{2})\otimes\cdots\otimes(u_{w_{1}}^{\ast}\xi_{1})\otimes\eta\right),
\end{eqnarray*}
which shows that $(u_{\mathbf{g}sr}u_{\mathbf{g}st}^{*})\,\eta\in\mathbf{Q}(u_{\mathbf{g}s}\Omega)\mathcal{H}_{\Gamma}$.

\item \emph{Case 2:} If $s_{1}=1$, then
\begin{eqnarray*}
 (u_{\mathbf{g}sr}u_{\mathbf{g}st}^{*})\,\eta
 &=& u_{\mathbf{g}sr}\left(u_{t}^{\ast}u_{s}^{\ast}u_{w_{L}}^{\ast}\cdots u_{w_{1}}^{\ast}\right)(\eta_{1}\otimes\cdots\otimes\eta_{n}) \\
 &=& u_{\mathbf{g}sr}\left(u_{t}^{\ast}u_{s}^{\ast}u_{w_{L}}^{\ast}\cdots u_{w_{2}}^{\ast}\right)
 \bigl((u_{w_{1}}^{\ast}\eta_{1}-\langle u_{w_{1}}\xi_{1},\eta_{1}\rangle\xi_{1})
 \otimes\eta_{2}\otimes\cdots\otimes\eta_{n}\bigr) \\
 &&\qquad \qquad \qquad  +\,\langle u_{w_{1}}\xi_{1},\eta_{1}\rangle \, u_{\mathbf{g}sr}\left(u_{t}^{\ast}u_{s}^{\ast}u_{w_{L}}^{\ast}\cdots u_{w_{2}}^{\ast}\right)(\eta_{2}\otimes\cdots\otimes\eta_{n}).
\end{eqnarray*}
The first summand lies in $\mathbf{Q}(u_{\mathbf{g}s}\Omega)\mathcal{H}_{\Gamma}$ by the argument of Case~1, so it remains to prove that the second summand is also contained in $\mathbf{Q}(u_{\mathbf{g}s}\Omega)\mathcal{H}_{\Gamma}$. Iterating this step implies that $(u_{\mathbf{g}sr}u_{\mathbf{g}st}^{*})\,\eta\in\mathbf{Q}(u_{\mathbf{g}s}\Omega)\mathcal{H}_{\Gamma}$ whenever $n\leq L+1$, while for $n>L+1$ it remains to prove
\[
\langle u_s \xi_1 ,\eta_{L+1} \rangle \left(\prod_{i=1}^{L}\langle u_{w_{i}}\xi_{1},\eta_{1}\rangle\right)u_{\mathbf{g}sr}u_t^\ast(\eta_{L+2}\otimes\cdots\otimes\eta_{n}) \in\mathbf{Q}(u_{\mathbf{g}s}\Omega)\mathcal{H}_{\Gamma}.
\]
But this vector equals zero since $\eta \in \ker(\mathbf{Q}(u_{\mathbf{g}st}\Omega))$, which completes the argument.
\end{itemize}

This finishes the proof of \eqref{eq:FreeMainIdentity}.\\

The proof of the identites \eqref{FreeMultiplicationIdentities(1.1)}--\eqref{FreeMultiplicationIdentities(1.5)}, \eqref{FreeMultiplicationIdentities(2.1)}--\eqref{FreeMultiplicationIdentities(2.2)}, and \eqref{FreeMultiplicationIdentities(3.1)} now follow exactly as in the proof of Lemma \ref{MultiplicationIdentities-1}. \end{proof}

We now proceed to prove Theorem~\ref{Selflessness2}.

\begin{proof}[Proof of Theorem~\ref{Selflessness2}]
Let $\mathcal{F} \subseteq W_{\Gamma} \setminus \{e\}$ be a finite set, and let $n \in \mathbb{N}$ be a positive integer.  
By Proposition~\ref{FreeDifferentPaths}, there exist distinct closed walks $(w_{1}^{(1)}, \ldots, w_{L}^{(1)}), \ldots, (w_{1}^{(n)}, \ldots, w_{L}^{(n)})$ in $\widetilde{\Gamma}^{c}$ covering the whole graph, such that $s = w_{1}^{(1)} = \cdots = w_{1}^{(n)}$, $t = w_{L}^{(1)} = \cdots = w_{L}^{(n)}$, and satisfying $\mathbf{Q}(u_{\mathbf{g}_{i}}\Omega)\, a \, \mathbf{Q}(u_{\mathbf{g}_{j}}\Omega) = 0$ for all $1 \le i,j \le n$ and every reduced operator $a \in \mathbf{A}_{\Gamma}$ of type $\mathbf{w} \in \mathcal{F}$, where $\mathbf{g}_{i} := w_{1}^{(i)} \cdots w_{L}^{(i)}$ and $\mathbf{g}_{j} := w_{1}^{(j)} \cdots w_{L}^{(j)}$.

Define
\[
T_{n,\mathcal{F}}
:= \frac{1}{\sqrt{2n}} 
\sum_{i=1}^{n} 
\left(
  \mathbf{Q}(u_{\mathbf{g}_{i}st}\Omega)\,(u_{\mathbf{g}_{i}st}u_{\mathbf{g}_{i}sr}^{*})
  + (u_{\mathbf{g}_{i}sr}u_{\mathbf{g}_{i}st}^{*})\,\mathbf{Q}(u_{\mathbf{g}_{i}st}\Omega)^{\perp}
\right)
\in \mathcal{B}(\mathcal{H}_{\Gamma}),
\]
and note that $T_{n,\mathcal{F}} + T_{n,\mathcal{F}}^{*}
= \frac{1}{\sqrt{2n}} 
\sum_{i=1}^{n}
\left(u_{\mathbf{g}_{i}st}u_{\mathbf{g}_{i}sr}^{*} + u_{\mathbf{g}_{i}sr}u_{\mathbf{g}_{i}st}^{*}\right)
\in \mathbf{A}_{\Gamma}$.

By Lemma~\ref{FreeMultiplicationIdentities}\,(1), we have
\begin{eqnarray*}
T_{n,\mathcal{F}}^{\ast}T_{n,\mathcal{F}} & = & \frac{1}{2n}\sum_{i,j=1}^{n}\left((u_{\mathbf{g}_{i}sr}u_{\mathbf{g}_{i}st}^{\ast})\mathbf{Q}(u_{\mathbf{g}_{i}st}\Omega)+\mathbf{Q}(u_{\mathbf{g}_{i}st}\Omega)^{\perp}(u_{\mathbf{g}_{i}st}u_{\mathbf{g}_{i}sr}^{\ast})\right) \\
& & \qquad \qquad \times \left(\mathbf{Q}(u_{\mathbf{g}_{j}st}\Omega)(u_{\mathbf{g}_{j}st}u_{\mathbf{g}_{j}sr}^{\ast})+(u_{\mathbf{g}_{j}sr}u_{\mathbf{g}_{j}st}^{\ast})\mathbf{Q}(u_{\mathbf{g}_{j}st}\Omega)^{\perp}\right)\\
 & = & \frac{1}{2n}\sum_{i=1}^{n}\left((u_{\mathbf{g}_{i}sr}u_{\mathbf{g}_{i}st}^{\ast})\mathbf{Q}(u_{\mathbf{g}_{i}st}\Omega)+\mathbf{Q}(u_{\mathbf{g}_{i}st}\Omega)^{\perp}(u_{\mathbf{g}_{i}st}u_{\mathbf{g}_{i}sr}^{\ast})\right) \\
 & & \qquad \qquad \times \left(\mathbf{Q}(u_{\mathbf{g}_{i}st}\Omega)(u_{\mathbf{g}_{i}st}u_{\mathbf{g}_{i}sr}^{\ast})+(u_{\mathbf{g}_{i}sr}u_{\mathbf{g}_{i}st}^{\ast})\mathbf{Q}(u_{\mathbf{g}_{i}st}\Omega)^{\perp}\right)\\
 & = & \frac{1}{2n}\sum_{i=1}^{n}\left((u_{\mathbf{g}_{i}sr}u_{\mathbf{g}_{i}st}^{\ast})\mathbf{Q}(u_{\mathbf{g}_{i}st}\Omega)(u_{\mathbf{g}_{i}st}u_{\mathbf{g}_{i}sr}^{\ast})+\mathbf{Q}(u_{\mathbf{g}_{i}st}\Omega)^{\perp}\right)\\
 & = & 1-\frac{1}{2n}\sum_{i=1}^{n}\left((u_{\mathbf{g}_{i}sr}u_{\mathbf{g}_{i}st}^{\ast})\mathbf{Q}(u_{\mathbf{g}_{i}st}\Omega)^{\perp}(u_{\mathbf{g}_{i}st}u_{\mathbf{g}_{i}sr}^{\ast})+\mathbf{Q}(u_{\mathbf{g}_{i}st}\Omega)\right).
\end{eqnarray*}
Arguing as in the proof of Theorem~\ref{Selflessness1}, we deduce from Lemma~\ref{FreeMultiplicationIdentities}\,(2) that $\|T_{n,\mathcal{F}}^{*}T_{n,\mathcal{F}} - 1\| \le \frac{1}{n}$. Furthermore, by Lemma~\ref{FreeMultiplicationIdentities}\,(3),
\begin{eqnarray*}
 &  & T_{n,\mathcal{F}}^{\ast}aT_{n,\mathcal{F}}\\
 & = & \frac{1}{2n}\sum_{i,j=1}^{n}\left((u_{\mathbf{g}_{i}sr}u_{\mathbf{g}_{i}st}^{\ast})\mathbf{Q}(u_{\mathbf{g}_{i}st}\Omega)+\mathbf{Q}(u_{\mathbf{g}_{i}st}\Omega)^{\perp}(u_{\mathbf{g}_{i}st}u_{\mathbf{g}_{i}sr}^{\ast})\right) \\
 & & \qquad \qquad \times a \left(\mathbf{Q}(u_{\mathbf{g}_{j}st}\Omega)(u_{\mathbf{g}_{j}st}u_{\mathbf{g}_{j}sr}^{\ast})+(u_{\mathbf{g}_{j}sr}u_{\mathbf{g}_{j}st}^{\ast})\mathbf{Q}(u_{\mathbf{g}_{j}st}\Omega)^{\perp}\right)\\
 & = & \frac{1}{2n}\sum_{i,j=1}^{n}\left((u_{\mathbf{g}_{i}sr}u_{\mathbf{g}_{i}st}^{\ast})\mathbf{Q}(u_{\mathbf{g}_{i}st}\Omega)+\mathbf{Q}(u_{\mathbf{g}_{i}st}\Omega)^{\perp}(u_{\mathbf{g}_{i}st}u_{\mathbf{g}_{i}sr}^{\ast})\right) \\
 & & \qquad \qquad \times \left(\mathbf{Q}(u_{\mathbf{g}_{i}}\Omega)a\mathbf{Q}(u_{\mathbf{g}_{j}}\Omega)\right) \\
 & & \qquad \qquad \qquad \times \left(\mathbf{Q}(u_{\mathbf{g}_{j}st}\Omega)(u_{\mathbf{g}_{j}st}u_{\mathbf{g}_{j}sr}^{\ast})+(u_{\mathbf{g}_{j}sr}u_{\mathbf{g}_{j}st}^{\ast})\mathbf{Q}(u_{\mathbf{g}_{j}st}\Omega)^{\perp}\right)\\
 & = & 0
\end{eqnarray*}
for every reduced operator $a \in \mathbf{A}_{\Gamma}$ of type $\mathbf{w} \in \mathcal{F} \setminus \{e\}$. The desired conclusion now follows by the same argument as in the proof of Theorem~\ref{Selflessness1}.
\end{proof}

\begin{example} \label{Example}
Let $\mathrm{C}^\ast(u,v) \cong \mathrm{C}^\ast(\mathbb F_2)$ be the universal $\mathrm{C}^\ast$-algebra generated by two unitaries $u$ and $v$.
Choi proved in \cite[Theorem 7]{Choi} that $\mathrm{C}^\ast(\mathbb F_2)$  admits a countable separating family of
representations $(\pi_n)_{n\in \mathbb{N}_{\geq 1}}$ such that $\pi_n$ has dimension $2n$ for each $n$ (see the proof of \cite[Theorem 7]{Choi}). Set $S=\{1,-1,i,-i\}$, and for each $\eta \in S$ let $\gamma_\eta$ be the automorphism of $\Cs(\mathbb F_2)$ induced by
$\gamma_\eta(u)=u$ and $\gamma_\eta(v)=\eta v$. 
Next, define a tracial state on $\Cs(\mathbb F_2)$ as follows,
\[
\tau_2(x)
   :=\frac{1}{4}\sum_{\eta\in S}\sum_{n\ge1} 2^{-n} \,
       \text{tr}_{2n}\!\left(\pi_n(\gamma_\eta(x))\right)
\]
for all $x\in\Cs(\mathbb F_2)$, where $\text{tr}_{2n}$ denotes the tracial state on $M_{2n}$. Since the family of representations $(\pi_n)_n$ is separating, and each trace $\text{tr}_{2n}$ faithful, it follows that $\tau_2$ is faithful as well. (This is a similar argument as in \cite[Corollary 9]{Choi}.) 
Moreover, $\tau_2(v)=\tau_2(v^2)=0$ as $\sum_{\eta \in S}\eta=\sum_{\eta \in S}\eta^2=0$.
In particular, $(\Cs(\mathbb F_2),\tau_2)$ admits unitaries with  the Avitzour condition and $(\Cs(\mathbb F_2),\tau_2)\star(\Cs(\mathbb F_2),\tau_2)$ is completely selfless by Theorem \ref{Selflessness2}.  

More generally, as Choi remarks in \cite{Choi}, one can define a faithful tracial state on $\mathrm{C}^\ast(\mathbb F_n)$ for all $n\geq2$.
Using the same argument as before, for each $n\geq2$ we find a faithful tracial state $\tau_n$ on $\Cs(\mathbb F_n)$ admitting unitaries as above.
Therefore,  for $m,n\geq2$ we have that the reduced free product $(\mathrm{C}^\ast(\mathbb F_m),\tau_m)\star(\mathrm{C}^\ast(\mathbb F_n),\tau_n)$ is completely selfless. Moreover, if one replaces  $(\Cs(\mathbb F_n),\tau_n)$ by any \Cs-probability space $(A,\omega)$ admitting a unitary in $\ker(\omega)\cap A^{\omega}$, the same conclusion remains valid. 

\end{example}



\section*{Acknowledgments}

F.F.\ gratefully acknowledges support from the Simons Foundation Dissertation Fellowship SFI-MPS-SDF-00015100. He thanks Professors Ben Hayes and Srivatsav Kunnawalkam Elayavalli for their guidance and for all the interesting discussions surrounding the present topic.
M.\'O.C.\ gratefully acknowledges support from the National Science Foundation under Grant No.\ NSF DMS-2144739. He thanks Professor Ben Hayes, his advisor and principal investigator of the grant, for his invaluable guidance and support.
M.P.\ was supported by the Carlsberg Foundation grant CF24-1144. 
The authors would like to thank N.\ Ozawa for helpful comments on the first draft of this article. Furthermore, the authors also wish to thank the anonymous referees for providing insightful comments that led to an improved version of the article.




\begin{thebibliography}{10}

\bibitem{AmrutamGaoElayavalliPatchell24} T. Amrutam, D. Gao, S. Kunnawalkam Elayavalli, G. Patchell, \emph{Strict comparison in reduced group \Cs-algebras}, Invent. Math. 242 (2025), pp. 639--657.


\bibitem{Avitzour82} D. Avitzour, \emph{Free products of \Cs-algebras}, Trans. Amer. Math. Soc. 271 (1982), no. 2, pp. 423--435.

\bibitem{BarlakSzabo} S. Barlak, G. Szabó, \emph{Sequentially split $\ast$-homomorphisms between \Cs-algebras}, Int. J. Math. 27 (2016), no. 13, article 1650105.

\bibitem{Blackadar88} B. Blackadar, \emph{Comparison theory for simple \Cs-algebras}, in {\it Operator algebras and applications, Vol.\ 1} (1988), pp. 21--54, London Math. Soc. Lecture Note Ser., 135, Cambridge Univ. Press, Cambridge. 


\bibitem{BlanchardDykema01} E. Blanchard, K. Dykema, \emph{Embeddings of reduced free products of operator algebras}, Pacific J. Math. 199 (2001), no. 1, pp. 1--19.



\bibitem{BrownOzawa08} N. Brown, N. Ozawa, \emph{\Cs-algebras and finite-dimensional approximations}, Graduate Studies in Mathematics, 88. American Mathematical Society, Providence, RI, 2008. xvi+509 pp.


\bibitem{CarrionGabeSchafhauserTikuisisWhite23} J. R. Carrión, J. Gabe, C. Schafhauser, A. Tikuisis, S. White, \emph{Classifying $\ast $-homomorphisms I: Unital simple nuclear $C^*$-algebras}, (2023) URL \url{https://arxiv.org/abs/2307.06480}, preprint.


\bibitem{CaspersFima17} M. Caspers, P. Fima, \emph{Graph products of operator algebras}, J. Noncommut. Geom. 11 (2017), no. 1, pp. 367--411.

\bibitem{CaspersKlisseLarsen21} M. Caspers, M. Klisse, N. S. Larsen, \emph{Graph product Khintchine inequalities and Hecke \Cs-algebras: Haagerup inequalities, (non)simplicity, nuclearity and exactness}, J. Funct. Anal. 280 (2021), no. 1, article 108795.

\bibitem{Choi} M.-D. Choi, \emph{The full \Cs-algebra of the free group on two generators}, Pacific J. Math. 87 (1980), no. 1, pp. 41--48.

\bibitem{Cuntz78} J. Cuntz, \emph{Dimension Functions on Simple \Cs-algebras}. Math. Ann. 233 (1978), pp. 145--153.


\bibitem{Davis08} M. W. Davis, \emph{The Geometry and Topology of Coxeter groups}, London Mathematical Society Monographs Series, 32 (2008), Princeton University Press, Princeton, NJ.

\bibitem{DykemaRordam} K. J. Dykema, M. Rørdam, \emph{Projections in free product \Cs-algebras}, Geom. Funct. Anal. 8 (1998), no. 1, pp. 1--6.

\bibitem{DykemaRordamB} K. J. Dykema, M. Rørdam, \emph{Projections in free product \Cs-algebras}, II, Math. Z. 234 (2000), no. 1, pp. 103--113.

\bibitem{DykemaHaagerupRordam} K. J. Dykema, U. Haagerup, M. Rørdam, \emph{The stable rank of some free product \Cs-algebras}, Duke Math. J. 90 (1997), no. 1, pp. 95--121.

\bibitem{GoldbringSinclair17} I. Goldbring, T. Sinclair. \emph{Robinson forcing and the quasidiagonality problem}, Int. J. Math. 28 (2017), no. 2, article 1750008.

\bibitem{GongLin20} G. Gong, H. Lin, \emph{On classification of non-unital simple amenable
\Cs-algebras, II}, J. Geom. Phys. 158 (2020), article 103865.

\bibitem{GongLin22} G. Gong, H. Lin, \emph{On classification of nonunital amenable simple
\Cs-algebras, III: The range and the reduction}, Ann. K-Theory 7 (2022),
no. 2, pp. 279–384.

\bibitem{GongLinNiu20} G. Gong, H. Lin, Z. Niu, \emph{A classification of finite simple amenable $\mathcal Z$-stable \Cs-algebras, I: \Cs-algebras with generalized tracial rank one}, C. R. Math. Acad. Sci. Soc. R. Can. 42 (2020), no. 3, pp. 63--450.

\bibitem{GongLinNiu20b} G. Gong, H. Lin, Z. Niu, \emph{A classification of finite simple amenable $\mathcal Z$-stable \Cs-algebras, II: \Cs-algebras with rational generalized tracial rank one}, C. R. Math. Acad. Sci. Soc. R. Can. 42 (2020), no. 4, pp. 451--539.

\bibitem{Gould26} M. Gould, \emph{The selfless dichotomy}, (2026), URL \url{https://arxiv.org/pdf/2606.09654}, preprint.

\bibitem{Green90} E. R. Green, \emph{Graph products of groups}, (1990), PhD thesis, University of Leeds.


\bibitem{HayesElayavalliPatchellRobert25} B. Hayes, S. Kunnawalkam Elayavalli, G. Patchell, L. Robert, \emph{Selfless inclusions of \Cs-algebras}, (2025), URL \url{https://arxiv.org/abs/2510.13398}, preprint.


\bibitem{HayesElayavalliRobert25} B. Hayes, S. Kunnawalkam Elayavalli, L. Robert, \emph{Selfless reduced free product \Cs-algebras}, (2025), URL \url{https://arxiv.org/abs/2505.13265}, preprint.

\bibitem{Klisse23-1} M. Klisse, \emph{Topological boundaries of connected graphs and Coxeter groups}, J. Operator Theory 89 (2023), no.2, pp. 429--476.

\bibitem{Klisse23-2} M. Klisse, \emph{Simplicity of right-angled Hecke \Cs-algebras}, Int. Math. Res. Not. IMRN(2023), no.8, 6598--6623.

\bibitem{Klisse25} M. Klisse, \emph{Universal \Cs-algebras from graph products: structure and applications}, (2025), URL \url{https://arxiv.org/abs/2507.12271}, preprint.


\bibitem{Kumjian04} A. Kumjian, \emph{On certain Cuntz-Pimsner algebras}, Pacific J. Math. 217 (2004), no. 2, pp. 275–289.

\bibitem{ElayavalliPatchellTeryoshin25} S. Kunnawalkam Elayavalli, G. Patchell, L. Teryoshin, \emph{Some remarks on decay in countable groups and amalgamated free products}, (2025), URL \url{https://arxiv.org/abs/2509.08754}, preprint.

\bibitem{ElayavalliSchafhauser25} S. Kunnawalkam Elayavalli, C. Schafhauser, \emph{Negative resolution to the \Cs-algebraic Tarski problem}, (2025), URL \url{https://arxiv.org/abs/2503.10505}, preprint.

\bibitem{ElliottGongLinNiu15} G. A. Elliott, G. Gong, H. Lin, Z. Niu, \emph{On the classification of simple
amenable \Cs-algebras with finite decomposition rank, II}, J. Noncommut. Geom. 19 (2025), no. 1, pp. 73--104.

\bibitem{ElliottGongLinNiu20} G. A. Elliott, G. Gong, H. Lin, Z. Niu, \emph{The classification of simple
separable $KK$-contractible \Cs-algebras with finite nuclear dimension}, J.
Geom. Phys. 158 (2020), article 103861.

\bibitem{ElliottToms08} G. Elliott, A. Toms, \emph{Regularity properties in the classification program
for separable amenable \Cs-algebras}, Bull. Amer. Math. Soc. 45 (2008),
no. 2, pp. 229--245.

\bibitem{JiangSu} X. Jiang, H. Su, \emph{On a simple unital projectionless \Cs-algebra}, Amer. J. Math. 121
(1999), no. 2, pp. 359--413.


\bibitem{Lin25} H. Lin, \emph{Strict comparison and stable rank one}, J. Funct. Anal. 289 (2025), no. 9, article 111065.

\bibitem{MaWangYang25} X. Ma, D. Wang, W. Yang, \emph{Boundary actions of CAT(0) spaces: topological freeness and applications to \Cs-algebras}, (2025), URL \url{https://arxiv.org/abs/2510.05669}, preprint.

\bibitem{MatuiSato12} H. Matui and Y. Sato, \emph{Strict comparison and $\mathcal{Z}$-absorption of nuclear \Cs-algebras}, Acta Math. 209 (2012), no. 1, pp. 179--196.

\bibitem{Mlotkowski04} W. M\l otkowski, \emph{$\Lambda$-free probability}, Infin. Dimens. Anal. Quantum Probab. Relat. Top. 7 (2004), no. 1, pp. 27--41.


\bibitem{Ozawa25} N. Ozawa, \emph{Proximality and selflessness for group \Cs-algebras}, (2025), URL \url{https://arxiv.org/abs/2508.07938}, preprint.


\bibitem{RaumThielVilalta25} S. Raum, H. Thiel, E. Vilalta, \emph{Strict comparison for twisted group \Cs-algebras}, (2025), URL \url{https://arxiv.org/abs/2505.18569}, preprint.

\bibitem{Rieffel} M. Rieffel, \emph{Dimension and stable rank in the $K$-theory of \Cs-algebras}, Proc. London
Math. Soc. 46 (1983), no. 3, pp. 301--333.

\bibitem{Robert12} L. Robert, \emph{Classification of inductive limits of 1-dimensional NCCW complexes}, Adv. Math. 231 (2012), no. 5, pp. 2802--2836.

\bibitem{Robert25} L. Robert, \emph{Selfless \Cs-algebras}, Adv. Math. 478 (2025), article 110409.

\bibitem{Rordam04} M. R\o rdam, \emph{The stable and the real rank of $\mathcal Z$-absorbing \Cs-algebras}, Int. J. Math. 15 (2004), no. 10, pp. 1065--1084.


\bibitem{Schafhauser20} C. Schafhauser, \emph{Subalgebras of simple AF-algebras}, Ann. Math. 192
(2020), no. 2, pp. 309--352.

\bibitem{SpeicherWysoczanski16} R. Speicher, J. Wysocza\'{n}ski, \emph{Mixtures of classical and free independence}, Arch. Math. (Basel) 107 (2016), no. 4, pp. 445--453.

\bibitem{Thiel24} H. Thiel, \emph{Diffuse traces and Haar unitaries}, Amer. J. Math. 146 (2024), no. 5, pp. 1305--1337.

\bibitem{TikuisisWhiteWinter} A. Tikuisis, S. White, W. Winter, \emph{Quasidiagonality of nuclear \Cs-algebras}, Ann. Math. 185 (2017), no. 1, pp. 229--284.

\bibitem{Toms} A. S. Toms, \emph{On the classification problem for nuclear \Cs-algebras}, Ann. Math. 167 (2008), no. 3, 1029--1044.

\bibitem{Vigdorovich25} I. Vigdorovich, \emph{Structural properties of reduced \Cs-algebras associated with higher-rank lattices}, (2025), URL \url{https://arxiv.org/abs/2503.12737}, preprint.

\bibitem{Voiculescu85} D. Voiculescu, \emph{Symmetries of some reduced free product \Cs-algebras}, Operator algebras and their connections with topology and ergodic theory (Bu\c{s}teni, 1983), pp. 556--588. Lecture Notes in Math., 1132 (1985) Springer-Verlag, Berlin. 


\bibitem{Winter10} W. Winter, \emph{Decomposition rank and Z-stability}, Invent. Math. 179 (2010),
no. 2, pp. 229–301.

\bibitem{Winter12} W. Winter, \emph{Nuclear dimension and $\mathcal{Z}$-stability of pure \Cs-algebras}, Invent. Math. 187 (2012), no. 2, 259--342.

\bibitem{WinterZacharias10} W. Winter, J. Zacharias, \emph{The nuclear dimension of \Cs-algebras}, Adv.
Math. 224 (2010), no. 2, pp. 461–498.

\end{thebibliography}
\end{document}